\documentclass{article}

\usepackage{iclr2026_conference}
\usepackage{amsthm,amssymb} 
\usepackage{algpseudocode}
\usepackage[linesnumbered,ruled,vlined]{algorithm2e}

\usepackage{comment}
\usepackage{subcaption}

\usepackage{amsmath,amsfonts,bm}

\def\eqref#1{equation~\ref{#1}}
\def\1{\bm{1}}

\DeclareMathAlphabet{\mathsfit}{\encodingdefault}{\sfdefault}{m}{sl}
\SetMathAlphabet{\mathsfit}{bold}{\encodingdefault}{\sfdefault}{bx}{n}

\DeclareMathOperator*{\argmin}{arg\,min}

\newcommand{\gf}{\nabla f}

\newcommand{\gh}{\nabla h}
\newcommand{\dual}[1]{\left\langle {#1} \right\rangle}

\newtheorem{remark}{Remark}[section]

\newtheorem{theorem}{Theorem}[section]
\newcommand{\bs}{\boldsymbol}

\newcommand{\nm}[1]{\left\lVert {#1} \right\rVert}

\newcommand{\ssnm}[1]
{
  \left\vert\kern-0.25ex
  \left\vert\kern-0.25ex
  \left\vert
  {#1}
  \right\vert\kern-0.25ex
  \right\vert\kern-0.25ex
  \right\vert
}

\newcommand{\LC}[1]{\textcolor{blue}{#1}}

\newcommand{\RV}[1]{#1}

\newtheorem{lemma}[theorem]{Lemma}

\newtheorem{proposition}[theorem]{Proposition}

\usepackage[utf8]{inputenc} % allow utf-8 input
\usepackage[T1]{fontenc}    % use 8-bit T1 fonts
\usepackage{hyperref}       % hyperlinks
\usepackage{url}            % simple URL typesetting
\usepackage{booktabs}       % professional-quality tables
\usepackage{amsfonts}       % blackboard math symbols
\usepackage{nicefrac}       % compact symbols for 1/2, etc.
\usepackage{microtype}      % microtypography
\usepackage{xcolor}         % colors

\usepackage[hypcap=false]{caption} 
\usepackage{graphicx}

\usepackage{enumitem}

  \newcounter{mnote}
  \let\oldmarginpar\marginpar
    \renewcommand\marginpar[1]{\-\oldmarginpar[\raggedleft\footnotesize #1]%
    {\raggedright\footnotesize #1}}

\title{Adaptive Accelerated Gradient Descent Methods for Convex Optimization}

\author{%
  Zeyi Xu \\
  Department of Mathematics\\
  University of California, Irvine\\
  Irvine, CA 92697 \\
  \texttt{zeyix1@uci.edu} \\
  \And
 Long Chen \\
Department of Mathematics\\
University of California, Irvine\\
Irvine, CA 92697 \\
\texttt{chenlong@uci.edu} \\
}

\iclrfinalcopy

\begin{document}

\maketitle

\begin{abstract}
This work proposes A$^2$GD, a novel adaptive accelerated gradient descent method for convex and composite optimization. Smoothness and convexity constants are updated via Lyapunov analysis. Inspired by stability analysis in ODE solvers, the method triggers line search only when accumulated perturbations become positive, thereby reducing gradient evaluations while preserving strong convergence guarantees. By integrating adaptive step size and momentum acceleration, A$^2$GD outperforms existing first-order methods across a range of problem settings.
%Crucially, A$^2$GD requires no manual hyperparameter tuning, enhancing its practical applicability.
\end{abstract}

%\paragraph{Keywords} Convex optimization, composite convex optimization, adaptive gradient descent, accelerated gradient descent, adaptive momentum

% \tableofcontents

\section{Introduction}

% \LC{Add background}
% !TEX root =  A2GDiclr_2026.tex

In this paper, we study the convex optimization problem  
\begin{equation}\label{eq:opt_basic}
    \min_{x \in \mathbb{R}^d} f(x),
\end{equation}
where $f$ is $\mu$-strongly convex and $L$-smooth. When $\mu = 0$, we additionally assume $f$ is coercive so that a global minimizer exists.  
We also consider the composite convex problem  
\begin{equation}\label{eq:opt_composite}
    \min_{x \in \mathbb{R}^d} f(x) := h(x) + g(x),
\end{equation}
where $h$ is $L$-smooth and $g$ is convex, possibly non-smooth, with a proximal operator. 
%    This problem arise frequently in settings where non-smooth regularization $g$ promotes generalization or enforces structure such as sparsity~\citep{Yin15_l1-2,zhang2015survey}.
%We make the following assumption on coerciveness.
%\begin{assumption}\label{assumption:coercivity}
%    The objective function $f$ is coercive. That is, the region where the function value of $f$ is less than or equal to $f(x_0)$ is bounded:
%    \[\mathrm{diam}(\{x\in\R^d|f(x)\leq f(x_0)\})\leq R,\]
%where the upper bound $R$ could be related to $x_0$.
%\end{assumption}

% The crucial parameters $L$ and $\mu$ are unknown to users. We shall develop Adaptive Accelerated Gradient Descent (A$^2$GD) Methods with adaptive choice of those two key parameters and thus further accelerate the accelerated gradient methods. 

First-order methods, which rely only on gradient information, are widely used in machine learning for their efficiency and scalability~\citep{BottouCurtisNocedal2018}. Among them, \textit{gradient descent} (GD), defined by
\begin{equation}\label{eq:GD}
    x_{k+1} = x_k - \alpha_k \nabla f(x_k),
\end{equation}
is fundamental. Despite its simplicity, GD faces two main challenges:

\begin{itemize}[leftmargin=10pt, topsep=0pt, partopsep=0pt]
    \item \textbf{Step size selection.} Convergence depends heavily on the step size $\alpha_k$. Small $\alpha_k$ slows progress; large $\alpha_k$ risks divergence. For $L$-smooth functions, $\alpha_k = 1/L$ is standard, but this global constant often mismatches local curvature.
    
    \item \textbf{Slow convergence.} Even with an optimal step size, GD is slow on ill-conditioned problems, i.e., when $L/\mu \gg 1$.
\end{itemize}

We briefly review strategies addressing these issues:

%\paragraph{Backtracking line search}
%This approach begins with a large step size estimate, reducing it until a condition such as the Armijo--Goldstein criterion is met. Nesterov~\citep{Nesterov2012GradientMF} extended this to composite objectives. While it yields near-optimal steps with theoretical guarantees, it requires additional gradient evaluations per iteration, increasing cost.
%
%Classical backtracking line search methods were introduced by Armijo~\citep{Armijo1966}, Goldstein~\citep{Goldstein1962}, and Wolfe~\citep{Wolfe1969}. Nesterov~\citep{Nesterov2012GradientMF} extended line search to accelerated proximal gradient methods, achieving an accelerated sublinear convergence rate of~$O(1/k^2)$ for convex composite objectives. Subsequent works~\citep{Ito2021} and~\citep{Liu2017} build upon this framework to address regularized composite problems using line search-based accelerated methods. \citep{Carderera2021} another variant of accelerated gradient method with parameters determined by line search. Line search method is simple and theoretically sound, but it involves ($3-4$) extra function/gradient evaluation per iteration, making the algorithm computationally expensive. More recently, an adaptive line search strategy proposed in~\citep{cavalcanti2025adaptive} has demonstrated improved efficiency by reducing the number of backtracking steps. 

\vskip -12pt

%\LC{, while Carderera et al.~\citep{Carderera2021} introduce a variant with parameters tuned via line search.} \mnote{check this reference. not mentioned here but later.}
\paragraph{Adaptive step sizes}  
Adaptive schemes such as the Barzilai--Borwein (BB) method~\citep{Barzilai;Borwein:1988Two-Point} estimate step sizes from past iterates:  
\begin{equation}\label{eq:BB}
\alpha_k = \frac{\langle x_k - x_{k-1}, \nabla f(x_k) - \nabla f(x_{k-1}) \rangle}{\| \nabla f(x_k) - \nabla f(x_{k-1}) \|^2}, 
\end{equation}
with low computational overhead. However, BB-type methods are heuristic, and may diverge even for simple convex problems~\citep{Burdakov;Dai;Huang:2019Stabilized}; guarantees are largely limited to quadratic cases~\citep{Dai-BBconv}. Extensions~\citep{Zhou2006-adaptive,Dai2015-BB} improve robustness but still lack general theory.

Polyak’s method~\citep{POLYAK1969}, foundational to adaptive approaches such as AdaGrad and AMSGrad~\citep{Vaswani2020AdaptiveGM}, ensures convergence but requires the optimal value $f^*$, which is rarely available. 

\paragraph{Acceleration}  
Momentum-based methods accelerate convergence by leveraging past updates. The heavy-ball method~\citep{polyak1964some} and Nesterov’s accelerated gradient (NAG)~\citep{nesterov2003introductory} achieve the optimal rate $1 - \sqrt{\mu/L}$ under strong convexity, assuming known $L$ and $\mu$. In the convex case ($\mu = 0$), NAG with step size $1/(k+3)$~\citep{nesterov1983method} achieves the optimal $O(1/k^2)$ rate. Nesterov later extended this framework to composite problems by incorporating line search into accelerated proximal methods~\citep{Nesterov2012GradientMF}, also attaining $O(1/k^2)$.
%Adaptive acceleration has been explored to mitigate parameter sensitivity. Rashidi et al.~\citep{Rashidi} introduced adaptive momentum updates for the heavy-ball method, albeit with fixed step sizes and no guarantees. Related work in~\citep{Meng2011AcceleratingNM} and ~\citep{li2024simpleuniformlyoptimalmethod} improves NAG via adaptive parameter updates.
\RV{Despite the effectiveness, NAG are known to suffer from oscillations. Restarting is used to mitigate this issue \citet{o2015adaptive}.}

\paragraph{Backtracking line search}  
Backtracking line search begins with a large step size $\alpha_k$ and reduces it until conditions such as the Armijo–Goldstein criterion~\citep{Armijo1966,Goldstein1962} or Wolfe condition~\citep{Wolfe1969} are satisfied. Extensions~\citep{Ito2021,Liu2017} adapt line search to composite settings. \citet{guminov2019} update parameters in NAG with backtracking, while \citet{lan2023optimal} develop a parameter-free method that attains optimal bounds for convex problems, and the best known results for nonconvex settings. An adaptive variant~\citep{cavalcanti2025adaptive} reduces backtracking steps, improving efficiency. Despite robustness and simplicity, line search usually requires $3$-$4$ extra function or gradient evaluations per iteration, increasing cost.

\paragraph{Line-search free methods.}
Recent years have seen growing interest in line-search free adaptive methods. These algorithms keep the per-iteration cost of gradient descent while often achieving faster convergence and lower sensitivity to hyperparameters. \RV{\citet{levy2018online} incorporate AdaGrad-style adaptive step sizes into NAG and further extend to stochastic settings. However, their approach adapts only $L$, limiting its ability to go beyond standard NAG. }\citet{Malitsky;Mishchenko:2020Adaptive,Malitsky2023AdaptivePG} introduced adaptive proximal gradient methods with theoretical guarantees, though lack of acceleration can hinder performance on ill-conditioned problems. \citet{li2024simpleuniformlyoptimalmethod} and \citet{cavalcanti2025adaptiveaccelerationstrongconvexity} proposed adaptive NAG variants with backtracking-free updates, \RV{in which both $L$ and $\mu$ are adaptive enabling stronger numerical performances}.

%\textbf{In contrast, our method combines adaptivity with acceleration using a simple parameter update that \textit{mostly} does not require backtracking line search, outperforming existing line-search free methods both in theory and practice.}

In training deep neural networks, Adam (Adaptive Moment Estimation)~\citep{kingma2015adam} is a widely used optimization method that combines momentum with adaptive step sizes to stabilize and accelerate \RV{stochastic gradients}. However, the original Adam algorithm does not provide convergence guarantees, even in convex settings.
% until very recently~\citep{DereicJentze2024Convergence}. 

%Meng’s method requires knowledge of $L$ and $\mu$ to verify a condition enabling less conservative updates. Li establishes a sequence of parameter conditions ensuring $O(1/k^2)$ sublinear convergence without relying on exact problem constants.

\begin{center}
    \textsc{Contribution}
\end{center}

\begin{itemize}[leftmargin=10pt, topsep=0pt, partopsep=0pt]
    \item We develop A$^2$GD, an adaptive accelerated gradient method with provable accelerated linear convergence for smooth (\ref{eq:opt_basic}) and composite convex optimization (\ref{eq:opt_composite}). 

\item We adapt stability analysis from ODE solvers to reduce line search overhead, activating it only when accumulated perturbations are positive. The method is thus line-search reduced rather than line-search free (Fig.~\ref{fig:compare_adGD_adGDwithlinesearch}), and it outperforms existing line-search free methods in both theory and practice.

    \item We show numerically that A$^2$GD also consistently outperforms AGD variants (where a single A denotes either adaptivity or acceleration) and other methods combining adaptivity and acceleration.
\end{itemize}  
%
%adaptively estimates both $L$ and $\mu$, reducing gradient evaluations while preserving strong convergence guarantees. It 

\paragraph{Limitations and Extensions}
While A$^2$GD achieves adaptive acceleration with theoretical guarantees, these results rely on convexity, and extending the framework to nonconvex settings remains open. \RV{Empirically, the method still works once the iterate enters locally convex basins.} We demonstrate the success of A$^2$GD on a composite $\ell_{1\text{-}2}$ problem, where the nonconvex regularizer admits a closed-form proximal operator.

Although our line-search reduced methods adds little practical overhead, we do not yet have a nontrivial upper bound on the total number of activations of line search as variability in local curvature for general convex functions can produce irregular triggering patterns. \RV{Empirically, typically fewer than $10$, and almost all activations occur in the early phase.}

Another extension is the stochastic setting. Developing a stochastic variant of A$^2$GD that preserves both adaptivity and acceleration under variance conditions would broaden applicability to large-scale machine learning, providing a step toward a theoretical justification of the empirical success of Adam.

\paragraph{Background on convex functions}
Let $f:~\mathbb{R}^d \to \mathbb{R}$ be differentiable. The Bregman divergence between $x, y \in \mathbb{R}^d$ is defined as 
\[
D_f(y, x) := f(y) - f(x) - \langle \nabla f(x), y - x \rangle.
\]
The function $f$ is $\mu$-strongly convex if for some $\mu > 0$,
\[
D_f(y, x) \geq \frac{\mu}{2} \|y - x\|^2, \quad \forall x, y \in \mathbb{R}^d.
\]
It is $L$-smooth, for some $L> 0$, if its gradient is $L$-Lipschitz:
\[
\|\nabla f(y) - \nabla f(x)\| \leq L \|y - x\|, \quad \forall x, y \in \mathbb{R}^d.
\]
The condition number is defined by $\kappa = L / \mu$. Let $\mathcal{S}_{L,\mu}$ denote the class of all differentiable functions that are both $\mu$-strongly convex and $L$-smooth.

For $f \in \mathcal{S}_{L,\mu}$, the Bregman divergence satisfies~\citep{nesterov2003introductory}:
\begin{equation}\label{eq:cocov}
\frac{1}{2L} \|\nabla f(x) - \nabla f(y)\|^2 \leq D_f(x, y) \leq \frac{1}{2\mu} \|\nabla f(x) - \nabla f(y)\|^2, \quad \forall x, y \in \mathbb{R}^d.
\end{equation}
Taking $y = x^*$, where $x^*$ minimizes $f$ and $\nabla f(x^*) = 0$, yields:
\begin{equation}\label{eq:gE}
\|\nabla f(x)\|^2 \geq 2\mu (f(x) - f(x^*)), \quad \forall x \in \mathbb{R}^d.
\end{equation}

\section{Adaptive Gradient Descent Method}

We illustrate our main idea using gradient descent method (\ref{eq:GD}) and later extend it to accelerated gradient descent.  
The steepest descent step chooses  
\begin{equation}\label{eq:exact}
\alpha_k^* = \arg\min_{\alpha > 0} f(x_k - \alpha \nabla f(x_k)),
\end{equation}  
which entails solving a one-dimensional convex problem. While conceptually simple, this can be costly unless a closed form is available.  

For $L$-smooth functions, the fixed step size $\alpha_k = 1/L$ guarantees convergence, but is often overly conservative when local curvature is much smaller than $L$. To improve efficiency, we design step sizes that adapt to local geometry using $f(x_k)$ and $\nabla f(x_k)$.  

We shall estimate the local Lipschitz constant $L_k$ through Lyapunov analysis of the gradient descent method (\ref{eq:GD}). Consider the Lyapunov function 
\begin{equation}\label{eq:simpleLya}
    E_k = f(x_k) - f(x^{\star}),
\end{equation}
where $x^{\star} \in \arg\min f(x)$ and $f(x^{\star}) = \min f$. Expanding $f$ at $x_{k+1}$ gives
\begin{align*}
    E_{k+1} - E_k 
    &= f(x_{k+1}) - f(x_k) = \dual{\gf(x_{k+1}), x_{k+1} - x_k} - D_f(x_k, x_{k+1}) \\
    &= -\alpha_k \dual{\gf(x_{k+1}), \gf(x_k)} - D_f(x_k, x_{k+1}) \\
    &= -\frac{\alpha_k}{2}\|\gf(x_{k+1})\|^2 - \frac{\alpha_k}{2}\|\gf(x_k)\|^2 \\
    &\quad + \frac{\alpha_k}{2}\|\gf(x_{k+1}) - \gf(x_k)\|^2 - D_f(x_k, x_{k+1}).
\end{align*}
Applying (\ref{eq:gE}) to $\|\gf(x_{k+1})\|^2$ and rearranging yields
\begin{equation}\label{eq:basicanalysis}
\begin{aligned}
\left(1 + \mu \alpha_k\right) E_{k+1} 
\leq E_k - \frac{\alpha_k}{2}\|\gf(x_k)\|^2 
+ \frac{\alpha_k}{2}\|\gf(x_{k+1}) - \gf(x_k)\|^2
- D_f(x_k, x_{k+1}).
\end{aligned}
\end{equation}

%\subsection{Elementary Line Search Approach}
If we use a line search to choose a small enough $\alpha_k$ such that
\begin{equation}\label{eq:linesearch}
\alpha_k = \frac{1}{L_k} \leq \frac{2D_f(x_k, x_{k+1})}{\|\gf(x_{k+1}) - \gf(x_k)\|^2},
\end{equation}
then dropping the negative terms in (\ref{eq:basicanalysis}) gives the linear convergence
\[
E_{k+1} \leq (1 + \mu/L_k)^{-1} E_k.
\]

Since $\alpha_k = 1/L_k$, choosing a smaller $\alpha_k$ is equivalent to using a larger $L_k$. By (\ref{eq:cocov}), the criterion (\ref{eq:linesearch}) holds once $L_k \geq L$. Standard backtracking starts with an initial estimate of $L_k$ and increases it iteratively by a factor $r > 1$ until (\ref{eq:linesearch}) is satisfied. This procedure requires at most $\mathcal{O}\!\left(\left\lceil \log L/\log r \right\rceil\right)$ iterations. A more adaptive and efficient backtracking scheme was recently proposed in~\cite{cavalcanti2025adaptive}, which we adapt for our purposes and briefly recall below.

Rewriting the stopping criterion (\ref{eq:linesearch}) gives  
\begin{equation*}
    v = \frac{2L_k D_f(x_k, x_{k+1})}{\|\nabla f(x_{k+1}) - \nabla f(x_k)\|^2} \geq 1.
\end{equation*}  
If $v < 1$, the criterion is not satisfied. Instead of increasing $L_k$ by a fixed ratio, we update it as $L_k \gets rL_k / v$, where $r > 1$ is a base ratio (e.g., $r=3$). This adaptive scaling adjusts to the gap between the current condition and the stopping criterion, improving both efficiency and accuracy.  

Even with adaptive backtracking, line search introduces overhead because each update of $L_k$ requires a new evaluation of $\nabla f(x_{k+1})$, \RV{often the dominant cost in gradient-based methods}, and sometimes also $f(x_{k+1})$. To reduce this, recent work has increasingly focused on line-search–free adaptive schemes; see the literature review in the introduction.

Enforcing line-search free updates is often too rigid and restrictive. In contrast, we reduce the number of line-search steps, achieving comparable cost to line-search free methods. Our approach is inspired by stability analysis in ODE solvers. The following result can be easily established by induction.

\begin{lemma}[A variant of Lemma 5.7.1. in \cite{Gautschi:2011Numerical}] \label{lm:cumsum}
Let $\left\{E_k\right\}$ be a positive sequence satisfying  
$$
E_{k+1} \leq \delta_k (E_k+b_k), \quad k=0,1, \ldots, 
$$  
where $\delta_k>0$ and $b_k \in \mathbb{R}$. Then  
$$
\begin{aligned}
E_{k+1} &\leq \left(\prod_{i=0}^{k} \delta_i\right) E_0+p_k, \quad k=0,1, \ldots ,
\end{aligned}
$$
where the accumlated perturbation
$$
p_k=\sum_{i=0}^k \left (\prod_{j=i}^k \delta_j \right ) b_i, \quad \text{ satisfying } \quad p_k = \delta_k (p_{k-1} + b_k).
$$
\end{lemma}

%Recall the linear decay of Lyapunov function $E$ of gradient descent,
%\begin{align*}
%    E_{k+1}\leq &{}\left(1-\mu\alpha_k\right) E_k-\frac{\alpha_k}{2}\|\gf(x_{k+1})\|^2 \\&+\frac{\alpha_k}{2}\|\gf(x_{k+1})-\gf(x_k)\|^2
%    -D_f(x_k,x_{k+1}).
%\end{align*}
%Denote perturbation at step $k$
%\begin{equation}
%    b_k=-\frac{\alpha_k}{2}\|\gf(x_{k+1})\|^2 +\frac{\alpha_k}{2}\|\gf(x_{k+1})-\gf(x_k)\|^2
%    -D_f(x_k,x_{k+1}).
%\end{equation}
%We tower up the exponential decay by replacing function $E$ with the right hand side,
%\begin{align*}
%    E_{k+1}\leq{}& (1-\mu \alpha_k)E_k + b_k\\
%    \leq{} &(1-\mu \alpha_k)\left((1-\mu\alpha_{k-1})E_{k-1}+b_{k-1}\right) + b_k\\
%    ={}& (1-\mu\alpha_k)(1-\mu\alpha_{k-1})E_{k-1} + (1-\mu\alpha_k)b_{k-1}+b_k\\
%    \leq{}& \cdots\\
%    \leq{}& \prod_{i=0}^k (1-\mu\alpha_i)E_0+\sum_{i=0}^k \prod_{j=i+1}^k (1-\mu\alpha_j)b_i.
%\end{align*}

We use an adaptive gradient descent method (ad-GD) to illustrate our main idea and refer to Appendix~A for the detailed algorithmic formulation. Applying Lemma~\ref{lm:cumsum} to GD under the Lyapunov analysis~(\ref{eq:basicanalysis}) gives
$$
\begin{aligned}
\delta_k &= (1+\mu/L_k)^{-1}, \quad b_k = b_k^{(1)}+b_k^{(2)}, \quad \text{where }\\
b_k^{(1)} &= \frac{1}{2L_k}\|\nabla f(x_{k+1})-\nabla f(x_k)\|^2 - D_f(x_k,x_{k+1}), \qquad b_k^{(2)} = -\frac{1}{2L_k}\|\nabla f(x_{k})\|^2.
\end{aligned}
$$ 
%The strong convexity constant $\mu$ may be unknown in practice.

In the line-search criterion~(\ref{eq:linesearch}), $L_k$ is selected so that $b^{(1)}_k<0$, ensuring $b_k<0$ at each step. This pointwise condition is sufficient but not necessary. Instead, we activate line search only when $p_k>0$ and increase $L_k$ until $p_k\le 0$. Classical line search enforces $b_k<0$ in an $\ell_\infty$ sense, while our approach permits a weighted \RV{$\ell_1$} control. Early in the iteration, when $\|\nabla f(x_k)\|$ is large, the negative terms $b_k^{(2)}$ accumulate and offset later positives, reducing activations. Once $p_k\le 0$, exponential decay follows:
\[
E_{k+1}\le \prod_{i=0}^k \Bigl(1+\frac{\mu}{L_i}\Bigr)^{-1} E_0.
\]
Figures~\ref{fig:lyapunov_values} and~\ref{fig:compare_adGD_adGDwithlinesearch} illustrate this behavior.

\vspace{-2 pt}
\begin{minipage}[htdp]{0.485\textwidth}
    \centering
    \includegraphics[width=0.8\linewidth]{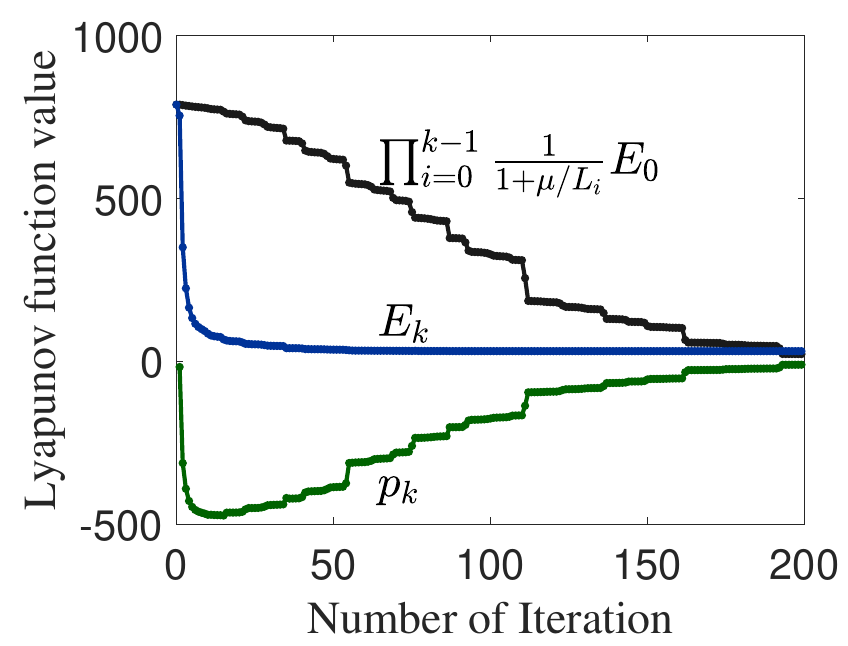}
    \captionsetup{hypcap=false}
    \captionof{figure}{The accumulated perturbation $p_k$ (green) stays negative and approaches zero. The Lyapunov values $E_k$ (blue) decay faster than the theoretical exponential rate $\left(\prod_{i=0}^{k}\delta_i\right)E_0$ (black). In the early iterations, $E_k$ decreases even more rapidly due to the large negative term $b_k^{(2)} = -\tfrac{1}{2L_k}\|\nabla f(x_k)\|^2$.}      \label{fig:lyapunov_values}
\end{minipage}%
\hfill
\begin{minipage}[htdp]{0.48\textwidth}
    \centering
    \includegraphics[width=0.8\linewidth]{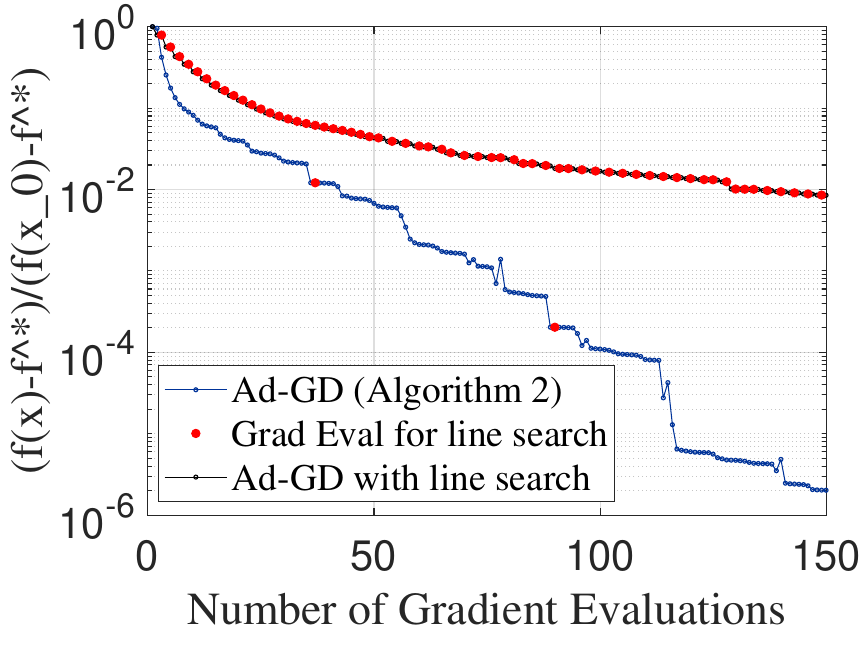}
    \captionsetup{hypcap=false}
    \captionof{figure}{For a logistic regression problem (\ref{eq:logreg}), gradient descent with line search enforcing $b_k^{(1)} \le 0$ (top curve) triggers backtracking every $3$--$4$ iterations on average. In contrast, ad-GD, which performs line search only when $p_k>0$, requires far fewer activations (bottom curve). Red dots mark iterations where line search is triggered.
} 
    \label{fig:compare_adGD_adGDwithlinesearch}
\end{minipage}

\vspace{2 pt}
\begin{theorem} \label{thm:agd_conv}
Assume $f\in \mathcal S_{L,\mu}$. Let $\{x_{k}\}$ be the sequence generated by gradient descent method (\ref{eq:GD}) with line search ensuring $p_k \leq 0$. Then we have
$$E_{k}\leq \prod_{i=0}^{k-1} \frac{1}{1+\mu/L_i}E_0 \leq \left (\frac{1}{1+ \mu/(c_rL)}\right )^k E_0.$$
\end{theorem}
\begin{proof}
As $p_k \leq 0$ for all $k$, linear convergence follows from~(\ref{eq:basicanalysis}). By~(\ref{eq:cocov}), the stopping criterion~(\ref{eq:linesearch}) is satisfied once $L_k \leq c_r L$ with at most $O(|\log L/\log r|)$ search steps, where $c_r \geq 1$ depends on the line-search scaling factor. Since $\mu_k \geq \mu$, the desired linear convergence rate follows.%Thus, at most $O(\log L)$ search steps are needed to find a suitable step size $L_k \leq L$.
\end{proof}

\begin{remark}\rm 
To improve efficiency, we set the next step size as  
$
\alpha_{k+1} = \frac{2D_f(x_k, x_{k+1})}{\|\nabla f(x_{k+1}) - \nabla f(x_k)\|^2}.
$
The gradient $\nabla f(x_{k+1})$ can be reused in the following gradient descent step. However, computing $D_f(x_k, x_{k+1})$ requires function evaluations $f(x_k)$ and $f(x_{k+1})$, which may be costly.  
To avoid these evaluations, we approximate $2D_f(x_k, x_{k+1})$ by its symmetrized form:
\[
2D_f(x_k, x_{k+1}) \approx D_f(x_k, x_{k+1}) + D_f(x_{k+1}, x_k)
= \langle \nabla f(x_{k+1}) - \nabla f(x_k),\, x_{k+1} - x_k \rangle.
\]
This reduces the ratio to the form used in the BB gradient method~(\ref{eq:BB}). In contrast to BB, convergence of ad-GD is guaranteed by enforcing $p_k \leq 0$.
\end{remark}

\begin{remark}\rm
There are several variants depending on how we define $\delta_k$ and split $b_k^{(1)}$ and $b_k^{(2)}$. For example, we can use $\delta_k = 1- \mu/L_k$, $b_k^{(2)} = 0$, and the rest is $b_k^{(1)}$. 
%, and 
%\begin{equation}
% b_k^{(1)} :=  \frac{1}{2L_k}\|\gf(x_{k+1})-\gf(x_k)\|^2 - D_f(x_k,x_{k+1})
%           - \frac{1}{2L_k}\|\gf(x_{k+1})\|^2. 
%\end{equation}
Then $b_k^{(1)} \leq 0 $ is equivalent to the criteria proposed by  
 \citep{Nesterov2012GradientMF}  (Appendix A). 
%    \begin{equation}\label{nestorov_stopping}
%        m_{L_k}(x_{k+1};x_k)\geq f(x_{k+1}),
%    \end{equation}
%    where $m_{L_k}(y;x)=f(x)+\dual{\gf(x),y-x}+\frac{L_k}{2}\|y-x\|^2$, $x_{k+1}=\argmin_{y} m_{L_k}(y;x_k)$. % \LC{add appendix to prove the equivalence.}
\end{remark}

%The following propositions shows that, for linear problems, linear convergence can be achieved with larger step sizes, and line search at all steps is not necessary.
%\begin{prop}
%    Assume $f(x)=\frac{1}{2}(x-x^{\star})A(x-x^{\star})$. If $L_k\geq\frac{\sqrt{5}-1}{2}L$, then $b_k^{(1)}+b_k^{(2)}\leq 0$, and line search is not triggered. In particular, we have the linear decay $$E_{k+1}\leq \frac{1}{1+\mu\alpha_k}E_k$$.
%\end{prop}

\iffalse
\LC{To be deleted.}
Alternatively, we define $b_k$ as
\begin{align*}
    b_k={}&-\frac{\alpha_k}{2}\|\gf(x_{k+1})\|^2 +\frac{\alpha_k}{2}\|\gf(x_{k+1})-\gf(x_k)\|^2
    \\&-\frac{1}{2}\dual{\gf(x_{k+1})-\gf(x_k),x_{k+1}-x_k}\\
    ={}&-\frac{\alpha_k}{2}\dual{\gf(x_{k+1}),\gf(x_k)}.
\end{align*}
This definition does not rely on the evaluation of $f(x_k)$. When the objective $f$ is a quadratic form, $\nabla^2 f$ is a constant symmetric positive-definite matrix everywhere, and
\begin{equation}
    D_f(x_{k+1},x_k)=D_f(x_k,x_{k+1})=\frac{1}{2}\dual{\gf(x_{k+1})-\gf(x_k),x_{k+1}-x_k},
\end{equation}
such $b_k$ coincides with the original definition.

\LC{Please check (11) in MM24 NeuralPS paper. Such $b_k < 0$. Something is wrong. Please use $\|\gf(x_{k})\|^2$ in $b_k$}
\fi

% !TEX root =  A2GDiclr_2026.tex

\section{Adaptive Accelerated Gradient Descent Method}
In this section, we apply our adaptive strategy to accelerated gradient methods. We derive an identity for the difference of the Lyapunov function and adaptively adjust $L_k$ and $\mu_k$ to ensure the accumulated perturbation is non-positive. 
\begin{comment}
The perturbation term $b_k^{(1)}$ remains unchanged, while $b_k^{(2)}$ includes additional terms and is used to adjust $\mu_k$. The contraction rate $\delta_k$ improves from $(1+\mu_k/L_k)^{-1}$ to $(1+\sqrt{\mu_k/L_k})^{-1}$.
\end{comment}

%\subsection{Convex optimization}\label{sec:A2GDconvex}
We will use the Hessian-based Nesterov accelerated gradient (HNAG) flow proposed in \cite{chen2019orderoptimizationmethodsbased}
\begin{equation}\label{eq:Hagf-intro}
	\left\{
	\begin{aligned}
		x' = {}&y-x-\beta\nabla f(x),\\
		y'={}&x - y -\frac{1}{\mu}\nabla f(x),
	\end{aligned}
	\right.
\end{equation}
where $\beta$ is a positive parameter. An implicit and explicit (IMEX) discretization of (\ref{eq:Hagf-intro}) is
%A Gauss-Seidel scheme without extra gradient step is considered below
\begin{equation}\label{eq:ex-HNAG}
	\left\{
	\begin{aligned}
		x_{k+1}-x_{k}={}& \alpha_k \left (y_{k}-x_{k+1}\right ) - \frac{1}{L_k}\nabla f(x_{k}),\\
 y_{k+1}-y_{k}={}&
		-\frac{\alpha_k}{\mu_k}\nabla f(x_{k+1}) + \alpha_k \left ( x_{k+1} - y_{k+1}\right ),
	\end{aligned}
	\right.
\end{equation}
where $\alpha_k>0$ is the time step size and $L_k := (\alpha_k\beta_k)^{-1}$. 
%To emphasize the dependence of $\mu_k$, we refine notation of Lyapunov function to $\mathcal E(\bs z; \mu)$. 
Denote by $\bs z=(x, y)^{\intercal}$. Introduce the Lyapunov function
$$
\mathcal E(\bs z; \mu):=f(x)-f(x^{*})+\frac{\mu}{2}\nm{y-x^{*}}^2.
$$
The proof of the following identity can be found in Appendix B. 
\begin{lemma}\label{lem:identityHNAG}We have the identity
 	\begin{equation*}
		%	\label{diff-Lk}
		\begin{split}
&(1+\alpha_k) \mathcal E(\bs z_{k+1}; \mu_{k}) - \mathcal E(\bs z_k; \mu_k) 
			\\
= {} & \frac{1}{2}\left ( \frac{\alpha_k^2}{\mu_k}  - \frac{1}{L_k}\right )\nm{\nabla f(x_{k+1})}_*^2 \quad ({\rm I})\\
 +& \frac{1}{2L_k}\| \nabla f(x_{k+1}) - \nabla f(x_k) \|^2 - D_f(x_{k}, x_{k+1}) \quad ({\rm II}) \\
-&\frac{1}{2L_k}\nm{\nabla f(x_k)}_*^2+ \frac{\alpha_k\mu_k}{2}\left (\nm{x_{k+1}-x^{*}}^2 - \frac{1}{\mu_k} D_f(x^{*}, x_{k+1}) - (1+\alpha_k)\nm{x_{k+1}-y_{k+1}}^2\right ) \ ({\rm III}).
		\end{split}
	\end{equation*}
\end{lemma}

We can simply set
%drop the negative term $- \alpha_k D_f(x^{*}, x_{k+1})$ in ({\rm I}) and set 
$\alpha_k = \sqrt{\frac{\mu_k}{L_k}}$
so that $({\rm I}) = 0$. To control ({\rm II}) and ({\rm III}), define perturbations
\begin{equation} 
\begin{aligned}
    b_k^{(1)} &= \frac{1}{2L_k}\|\gf(x_{k+1})-\gf(x_k)\|^2-D_f(x_k,x_{k+1}),\\
    b_k^{(2)} &= -\frac{1}{2L_k}\nm{\nabla f(x_k)}^2+ \frac{\alpha_k\mu_k}{2}\left (R_k^2 - (1+\alpha_k)\nm{x_{k+1}-y_{k+1}}^2 \right),\\
p_{k}&=\frac{1}{1+\alpha_k} \left(p_{k-1}+b_k^{(1)}+b_k^{(2)}\right),\quad \forall k\geq 1 \text{ and } p_0 = 0.    
\end{aligned}
\end{equation}
%Different from the line search-based method where we impose negativity of $b_k^{(1)}$ and $b_k^{(2)}$ at each step, here only impose negativity on their accumulated sum. 
%Define total perturbation 
%\begin{equation}
%\end{equation}

The term $b_k^{(1)}$ measures deviation from the Lipschitz condition and is used to adjust $L_k$, while $b_k^{(2)}$ measures deviation from the strong convexity assumption and is used to adjust $\mu_k$. To enforce the lower bound $\mu_k \geq \mu$ when $\mu > 0$, we introduce
\[
R_k^2 := \left(1 - \mu/\mu_k\right) R^2,
\]
using the inequality $D_f(x^*, x_{k+1}) \geq \tfrac{\mu}{2} \|x_{k+1} - x^*\|^2$ and an upper bound $R$ such that $\|x_{k+1} - x^*\|^2 \leq R^2$. If $\mu_k < \mu$, then $b_k^{(2)} < 0$ and no further reduction of $\mu_k$ is allowed. The parameter $\mu$ can be a conservative estimate of the true convexity constant and serves as a lower bound for $\mu_k$.

Line search is triggered only when $p_k > 0$. If $b_k^{(1)} > 0$, $L_k$ is updated using adaptive backtracking~\cite{cavalcanti2025adaptive}. If $b_k^{(2)} > 0$, the convexity is not strong enough to support a large step, so we reduce $\mu_k$. In the limiting case $\mu_k = 0$, we will have $b_k^{(2)} \leq 0$. 

To update $\mu_k$ more precisely, we solve $b_k^{(2)} = 0$, treating $L_k$ as known and using the fixed rule $\alpha_k = \sqrt{\mu_k / L_k}$ for the step size. The leading term in the second part of $b_k^{(2)}$ is $\alpha_k \mu_k R_k^2 = \mu_k^{3/2} R_k^2 / L_k^{1/2}$, and the equation essentially reduces to a non-trivial scaling
\[
\frac{\mu_k^{3/2} R_k^2}{L_k^{1/2}} \approx \frac{\|\nabla f(x_k)\|^2}{L_k} 
\quad \Rightarrow \quad 
\mu_k \propto \frac{\| \nabla f(x_k) \|^{4/3}}{L_k^{1/3} R_k^{4/3}}.
\]
%Including additional negative terms in the update formula makes $\mu_k$ slightly larger, leading to a larger step size (learning rate) $\alpha_k$.
To preserve decay of the Lyapunov function, we enforce
\[
\mu_{k+1} \leq \mu_k \quad \Rightarrow \quad \mathcal{E}(\bs z_{k+1}; \mu_{k+1}) \leq \mathcal{E}(\bs z_{k+1}; \mu_k).
\]

To establish convergence guarantees, the parameter $\mu_k$ cannot decay too quickly. To control this decay, we follow the perturbation strategy of \cite{chen2025acceleratedgradientmethodsvariable} by introducing a parameter $\varepsilon$ and enforcing the lower bound $\mu_k \ge \varepsilon$. The value of $\varepsilon$ is halved only when certain decay conditions are met. Specifically, $\varepsilon$ is reduced if either $\mathcal{E}_k/\mathcal{E}_0 \le (R^2+1)\varepsilon/2$
or the number of iterations performed with the current $\varepsilon$ exceeds $m$. If $\mu>0$, the condition is reached within $\mathcal{O}(|\log \varepsilon|)$ iterations; if $\mu=0$, the iteration count for a fixed $\varepsilon$ is at most $m$.
Since $\mathcal{E}_k$ is not directly observable, we use the proxy $\|\nabla f(x_k)\|^2/\|\nabla f(x_0)\|^2$ for $\mathcal{E}_k/\mathcal{E}_0$.

\vskip -6pt
\begin{algorithm}[ht]
\caption{Adaptive Accelerated Gradient Method (A$^2$GD)}\label{alg:adaptiveHNAG}
\KwIn{$x_0, y_0 \in \mathbb{R}^n$, $L_0 > 0$, $\mu_0 > 0$, $R > 0$, $0< {\rm tol}\ll 1$, $\varepsilon>0$, $m\geq 1$}
% \For{$k = 0, 1, 2, \ldots$}{
\While{$\|\nabla f(x_k)\|> {\rm tol}\|\nabla f(x_0)\|$}{
    $\alpha_k \gets \sqrt{\mu_k / L_k}$\;
    $x_{k+1} \gets \frac{1}{\alpha_k + 1} x_k + \frac{\alpha_k}{\alpha_k + 1} y_k - \frac{1}{L_k (\alpha_k + 1)} \nabla f(x_k)$\;
    $y_{k+1} \gets \frac{\alpha_k}{\alpha_k + 1} x_{k+1} + \frac{1}{\alpha_k + 1} y_k - \frac{\alpha_k}{\mu_k (\alpha_k + 1)} \nabla f(x_{k+1})$\;
    $b_k^{(1)} \gets \frac{1}{2L_k} \| \nabla f(x_{k+1}) - \nabla f(x_k) \|^2 - D_f(x_k, x_{k+1})$\;
    $b_k^{(2)} \gets -\frac{1}{2L_k} \| \nabla f(x_k) \|_*^2 + \frac{\alpha_k \mu_k}{2} \left(R_k^2 - (1+\alpha_k)\|x_{k+1} - y_{k+1}\|^2 \right)$\;
    $p_k \gets \frac{1}{1+\alpha_k}(p_{k-1} + b_k^{(1)} + b_k^{(2)})$\;
    
    \If{$p_k > 0$}{
            \If{$b_k^{(1)} > 0$}{
                % $v \gets \frac{2L_k D_f(x_k,x_{k+1})}{\|\nabla f(x_{k+1}) - \nabla f(x_k)\|^2}$\;
                % $L_k\gets rL_k/v$\;
                $v\gets \frac{2L_k D_f(x_k,x_{k+1})}{\|\nabla f(x_{k+1}) - \nabla f(x_k)\|^2}$, $L_k\gets 3L_k/v$\;
            }
            \If{$b_k^{(2)} > 0$}{
                        $\mu_k \gets \max\left\{\varepsilon,\min\left\{\mu_k, \frac{\| \nabla f(x_k) \|^{4/3}}{L_k^{1/3} \left(R_k^2 - (1+\alpha_k) \|x_{k+1} - y_{k+1}\|^2\right)^{2/3}} \right\}\right\}$\;
            }
            Go to line 2\;
        }

    \Else{
        % $v \gets \frac{2L_k D_f(x_k,x_{k+1})}{\|\nabla f(x_{k+1}) - \nabla f(x_k)\|^2}$\;
        $L_{k+1}\gets \frac{\|\nabla f(x_{k+1}) - \nabla f(x_k)\|^2}{2D_f(x_k,x_{k+1})}$\;
        $\mu_{k+1} \gets \max\left\{\varepsilon,\min\left\{\mu_k, \frac{\| \nabla f(x_k) \|^{4/3}}{L_k^{1/3} \left(R_k^2 - (1+\alpha_k) \|x_{k+1} - y_{k+1}\|^2\right)^{2/3}} \right\}\right\}$\;
%        \textbf{break};
    }
% }
\If{\textbf{decay condition}}{$\varepsilon\gets \varepsilon/2$\;
$m \gets \lfloor \sqrt{2} \cdot m \rfloor + 1$\;}
$k\gets k+1$\;
}
\end{algorithm}

\vskip -4 pt
To ensure monotonic descent, updates with $f(x_{k+1}) > f(x_k)$ are rejected by setting $x_{k+1}=x_k$. When $\|y_k - x^\star\| \gg \|x_k - x^\star\|$, the Lyapunov function $\mathcal{E}$ may decrease mainly through $\|y_k - x^\star\|^2$, while $f(x_k)$ stagnates. To avoid this, we restart by setting $y_k = x_k$ if $f(x_k)$ fails to decrease for five consecutive iterations. These monitoring steps are omitted from Algorithm~\ref{alg:adaptiveHNAG} but are used in practice to improve stability. \RV{Restarting and accept/reject heuristics reduce oscillations but do not influence the main source of acceleration; see Fig.~\ref{fig:lor_nonadaptive_norestart}.}

To reduce sensitivity to initialization, we include a short warm-up phase using adaptive proximal gradient descent (AdProxGD)~\cite{Malitsky2023AdaptivePG} or ad-GD (Algorithm \ref{alg:adaptiveGD2e} in Appendix A). Starting from $x_0$, we run $10$ AdProxGD iterations and initialize A$^2$GD with $x_0=y_0:=x_{10}$, $\mu_0:=\min_{1\le k\le 10}\{L_k\}$, and $R = 100\,\|\nabla f(x_0)\|/\mu_0$. Although updating $R$ dynamically may help, the method is generally robust with a fixed $R$.
\RV{Our ablation study shows that A$^2$GD remains stable for perturbation up to a factor $1000$, indicating that the warm-up is not essential for good performance; see Fig. \ref{fig:ablation_warmup}.}

% \LC{these two algorithms are inconsistent. In Alg 2, it is $b_k^{(1)}\leq 0$ while in Alg 3, it is $p_k$. Please correct.}
% \LC{Find a way to merge Algorithm 3 into Algorithm 2. } 

%\iffalse
%\begin{algorithm}
%    \caption{Line search for $L_k$ and $\mu_k$}\label{alg:adaptiveHNAG-linesearch}
%    \begin{algorithmic}[1]
%        \WHILE{$b_k^{(1)}> 0$ or $b_k^{(2)}> 0$}
%            \STATE $\mu_k \gets \min\{\mu_k,\frac{\|\gf(x_k)\|^{4/3}}{L_k^{1/3}\left(R_k^2-(1+\alpha_k)\|x_{k+1}-y_{k+1}\|^2\right)^{2/3}}\}$, \quad $L_k \gets r L_k$
%            \STATE line 2 - 7 in Algorithm 2
%            ,\quad $\alpha_k\gets \sqrt{\frac{\mu_k}{L_k}}$
%            \STATE $\displaystyle{x_{k+1}\gets \frac{1}{\alpha_k+1}x_k+\frac{\alpha_k}{\alpha_k+1}y_k-\frac{1}{L_k(\alpha_k+1)}\gf(x_k)}$
%            \STATE $\displaystyle{y_{k+1}\gets \frac{\alpha_k}{\alpha_k+1}x_{k+1}+\frac{1}{\alpha_k+1}y_k-\frac{\alpha_k}{\mu_k(\alpha_k+1)}\gf(x_{k+1})}$
%            \STATE $b_k^{(1)} \gets \frac{1}{2L_k}\|\gf(x_{k+1})-\gf(x_k)\|^2-D_f(x_k,x_{k+1})$
%            \STATE $b_k^{(2)} \gets -\frac{1}{2L_k}\nm{\nabla f(x_k)}_*^2+ \frac{\alpha_k\mu_k}{2}\left (R_k^2 - (1+\alpha_k)\nm{x_{k+1}-y_{k+1}}^2 \right)$
%            \STATE $p_k \gets \frac{1}{1+\alpha_k}p_{k-1}+b_k^{(1)}+b_k^{(2)}$            
%        \ENDWHILE
%    \end{algorithmic}
%\end{algorithm}
%\fi

\begin{theorem}\label{thm:conv-ex1-ode-NAG}
Let $(x_k, y_k)$ be the iterates generated by the above algorithm. Assume function $f$ is $\mu$-strongly convex with $\mu\geq 0$. Let  $k_s$ be the total number of steps after halving $\varepsilon$ exactly $s$ times, i.e. $\varepsilon = 2^{-s}\varepsilon_0$. 
\setlist[enumerate]{leftmargin=20pt,labelsep=0.6em}
\begin{enumerate}
\item When $\mu=0$, ther exists a constant $C > 0$ so that
$$
\frac{\mathcal{E}_{k_s}}{\mathcal{E}_0} \leq \frac{R^2 + 1}{\left( C k_s + \varepsilon_0^{-1/2} \right)^2} = \mathcal{O}\left( \frac{1}{k_s^2} \right)
$$
So $\mathcal O(\sqrt{1/{\rm tol}})$  iteration steps can acheive $\mathcal{E}_{k_s}/\mathcal{E}_0\leq {\rm tol}$. 

\item When $\mu > 0$, the iteration number to achieve $\mathcal{E}_{k_s}/\mathcal{E}_0 \leq (R^2+1) 2^{-s}\varepsilon_0\leq {\rm tol}$ is bounded by $\mathcal O(\sqrt{L/\mu}\ \ln {\rm tol})$. 
\end{enumerate}
\end{theorem}

\section{Numerical Experiments}
%This section evaluates the performance of A$^2$GD on a range of optimization problems. 
We test A$^2$GD on smooth convex minimization tasks and compare it with several leading first-order methods, grouped into two categories: 
\begin{itemize}[leftmargin = 12pt, topsep=0pt, itemsep=0pt]
    \item \textbf{Accelerated but non-adaptive methods:} Nesterov’s accelerated gradient (NAG) with step size $1/(k+3)$~\citep{nesterov1983method}, accelerated over-relaxation heavy ball (AOR-HB)~\citep{wei2024accelerated}, and the triple momentum method (TM)~\citep{VanScoyFreemanLynch2018}. 
%    Since the classical heavy ball method lacks convergence guarantees~\cite{lessardAnalysisDesignOptimization2016}, we use AOR-HB, which ensures global accelerated convergence.
    \item \textbf{Adaptive methods:} adaptive proximal gradient descent (AdProxGD)~\citep{Malitsky2023AdaptivePG}, the Accelerated Adaptive Gradient Method (AcceleGrad)~\citep{levy2018online}% the auto-conditioned fast gradient method (AC-FGM)~\citep{li2024simpleuniformlyoptimalmethod}
    , and NAGfree~\citep{cavalcanti2025adaptive}.
\end{itemize}

For all examples, we set the tolerance to $\mathrm{tol} = 10^{-6}$ and use the stopping criterion $\|\nabla f(x_k)\| \leq \mathrm{tol} \cdot \|\nabla f(x_0)\|$. All experiments were run in MATLAB R2023a on a desktop with an Intel Core i5-7200U CPU (2.50 GHz) and 8 GB RAM. \RV{Because gradient evaluation dominates the computational cost, we report convergence in terms of gradient evaluations and mark additional evaluations from line search with red dots. The corresponding runtime comparisons are provided in Appendix~D and the performance is similar.}
%The stopping criterion is based on the relative norm of the gradient:
%\begin{equation}
%\|\nabla f(x_k)\| < \mathrm{tol} \cdot \|\nabla f(x_0)\|,
%\end{equation}
%where $\mathrm{tol}$ is a small positive constant.

\paragraph{Regularized Logistic Regression}
We report numerical simulations on a logistic regression problem with an $\ell_2$ regularizer:
\begin{equation}\label{eq:logreg}
    \min_{x \in \mathbb{R}^n} \left\{ \sum_{i=1}^m \log \big(1 + \exp(-b_i a_i^{\top} x)\big) + \frac{\lambda}{2}\|x\|^2 \right\},
\end{equation}
where $(a_i, b_i) \in \mathbb{R}^n \times \{-1, 1\}$ for $i=1,2,\ldots,m$. 

For this problem, $\mu = \lambda$ and $L = \lambda_{\max}\!\left(\sum_{i=1}^m a_i a_i^{\top}\right) + \lambda$. We use $(a_i, b_i)$ from the Adult Census Income dataset. After removing entries with missing values, the dataset contains 30,162 samples. The Lipschitz constant is $6.30 \times 10^4$. With regularization parameter $\lambda = 0.1$, the condition number is $\kappa = 6.30 \times 10^5$. 

\RV{In Fig.~\ref{fig:lor_nonadaptive_norestart}, we disable accept/reject and restarting in A$^2$GD and obtain A$^{2}$GD-plain. We compare it with other plain accelerated methods. In Fig.~\ref{fig:lor_nonadaptive_restart}, we equip all methods with comparable restarting schemes and evaluate them alongside A$^{2}$GD.} In Fig.~\ref{fig:lor_adaptive}, we compare A$^{2}$GD with other fine-tuned adaptive gradient methods using their recommended parameter settings. %Due to the varying landscape of the real dataset, adaptive methods (Fig.~\ref{fig:lor_adaptive}) converge faster than non-adaptive ones (Fig.~\ref{fig:lor_nonadaptive}).

 % d = 1000,$ and $n = 50$. As illustrated by Figure \ref{fig:log_reg_f}, this is an example where HB converges and it converges fastest. {However, such fast convergence lacks theoretical guarantees and may fail in cases like the one depicted in Figure \ref{fig:multiobj_err}. This highlights the importance of robust theoretical convergence analysis rather than relying on empirical success alone.} 
%Our AOR-HB method is comparable to NAG and ADR, and TM is slightly faster.

%\vspace{-10 pt}
\begin{minipage}[htdp]{0.31\textwidth}
    \centering
    \includegraphics[width=0.975\linewidth]{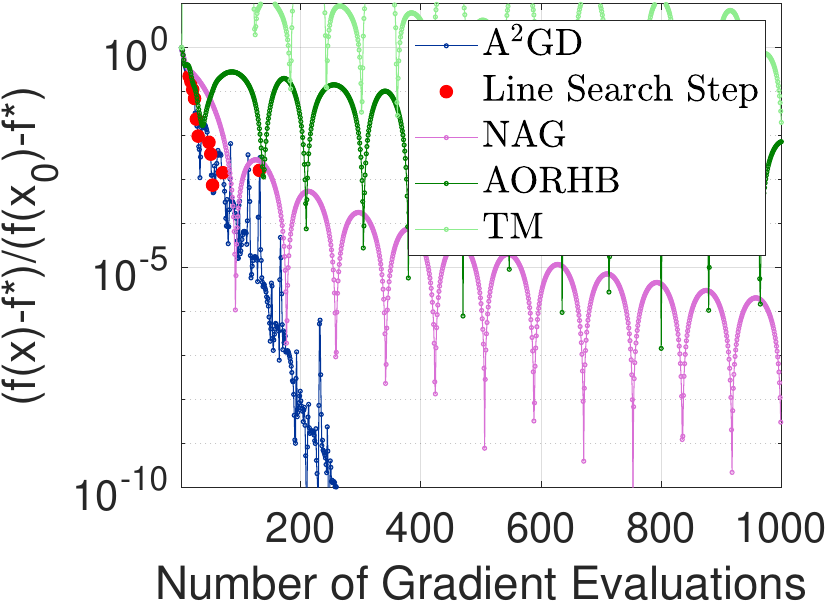}
    \captionsetup{hypcap=false}
    \captionof{figure}{\RV{Comparison without restarting.}}      \label{fig:lor_nonadaptive_norestart}
\end{minipage}%
\hfill
\begin{minipage}[htdp]{0.31\textwidth}
    \centering
    \includegraphics[width=0.975\linewidth]{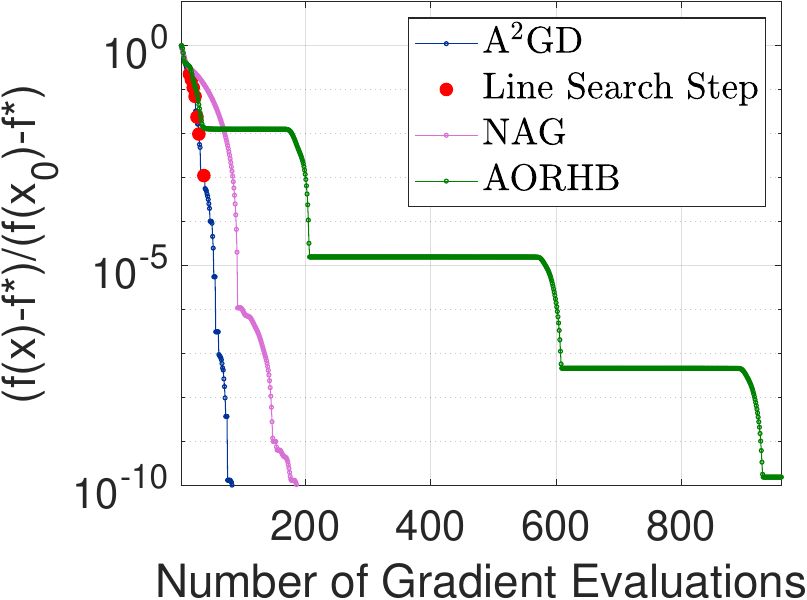}
    \captionsetup{hypcap=false}
    \captionof{figure}{\RV{Comparison with restarting.}} 
    \label{fig:lor_nonadaptive_restart}
\end{minipage}
\hfill
\begin{minipage}[htdp]{0.31\textwidth}
    \centering
    \includegraphics[width=0.95\linewidth]{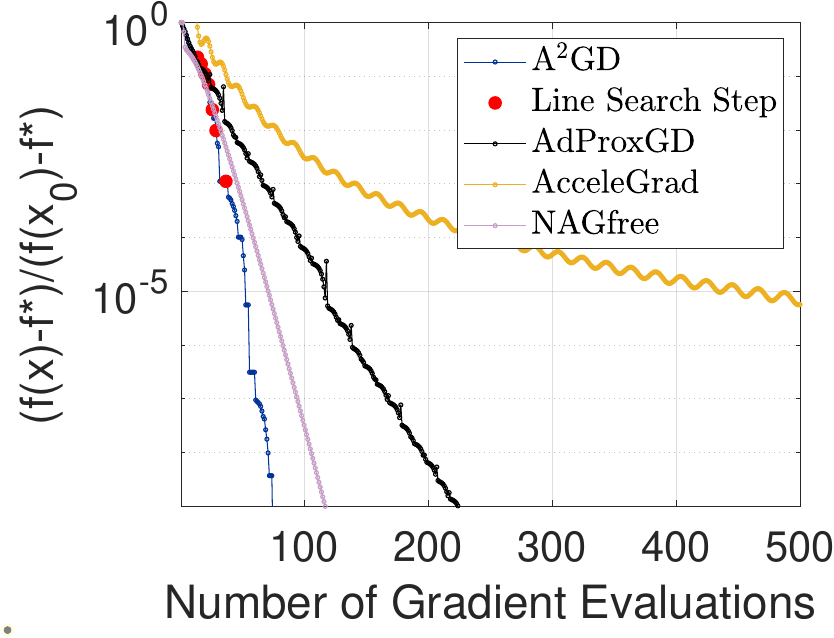}
    \captionsetup{hypcap=false}
    \captionof{figure}{\RV{Comparison with other adaptive methods.}} 
    \label{fig:lor_adaptive}
\end{minipage}

\paragraph{Maximum Likelihood Estimate of the Information Matrix}
We consider the maximum likelihood estimation problem from~\cite[(7.5)]{boyd2004convex}:
\begin{equation} \label{eq:MLEproblem}
    \begin{aligned}
        & \underset{X \in \mathbb{R}^{n \times n}}{\text{minimize}} \quad f(X) := -\log\det X + \mathrm{tr}(XY), \\
        & \text{subject to} \quad \lambda_{\min} \leq \lambda(X) \leq \lambda_{\max},
    \end{aligned}
\end{equation}
where $X$ is symmetric positive definite and $\lambda_{\min}, \lambda_{\max} > 0$ are given bounds. The condition number of $f$ is $\kappa=\lambda_{\max}^2/\lambda_{\min}^2$.
%Note that $\nabla^2 f(X)=X^{-1}\otimes X^{-1}$, where $\otimes$ is the Kronecker product, we derive 
% Note that the condition number $\kappa(f)$ may significantly exceed the ratio $\lambda_{\max}/\lambda_{\min}$.

%\vskip -10pt
%

Problem~(\ref{eq:MLEproblem}) has a composite structure, with a smooth term $f(X)$ and a nonsmooth indicator $g(X)$ enforcing spectral constraints. The proximal step for $g$ requires eigen-decomposition, eigenvalue projection, and matrix reconstruction, so gradient and proximal evaluations dominate the cost. Reducing these evaluations, particularly during backtracking, is therefore crucial. \RV{We again report convergence in terms of gradient steps and mark additional line-search evaluations with red dots, which occur very rarely after the initial stage and thus invisable in the figures.}

We extend A$^{2}$GD and its convergence analysis to the composite setting; details appear in Appendix~C. We compare A$^{2}$GD with several first-order proximal methods: AdProxGD~\citep{Malitsky2023AdaptivePG}, FISTA~\citep{beck2009fast}, and AOR-HB with perturbation~\citep{chen2025acceleratedgradientmethodsvariable}. We use tolerance $\mathrm{tol}=10^{-6}$ and adopt the stopping rule $\|\nabla h(x_k)+q_k\|\le \mathrm{tol}\,\|\nabla h(x_0)\|$ for all experiments.

%\vspace{-10 pt}
\begin{minipage}[htdp]{0.45\textwidth}
    \centering
    \includegraphics[width=0.75\linewidth]{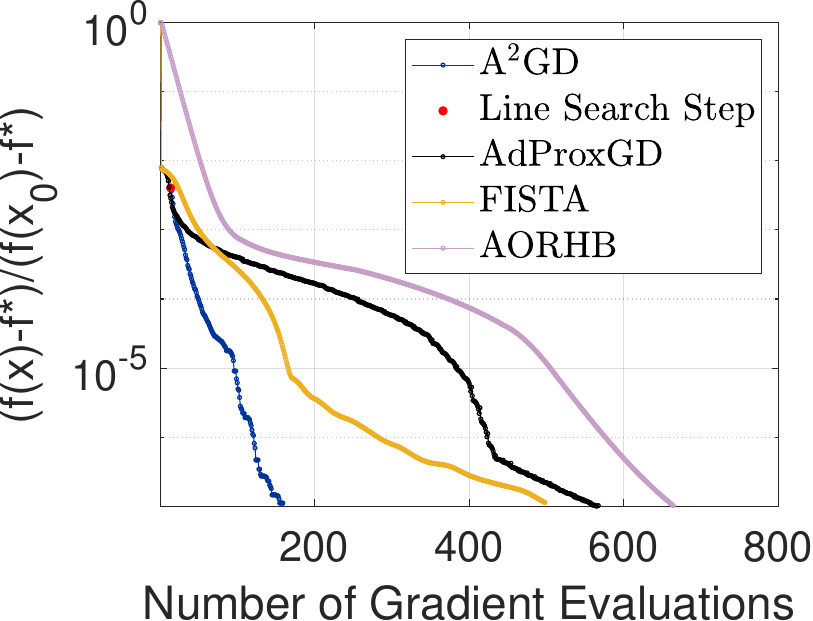}
    \captionsetup{hypcap=false}
    \captionof{figure}{Error curves under setting (1).}      \label{fig:mle_easy}
\end{minipage}%
\hfill
\begin{minipage}[htdp]{0.45\textwidth}
    \centering
    \includegraphics[width=0.75\linewidth]{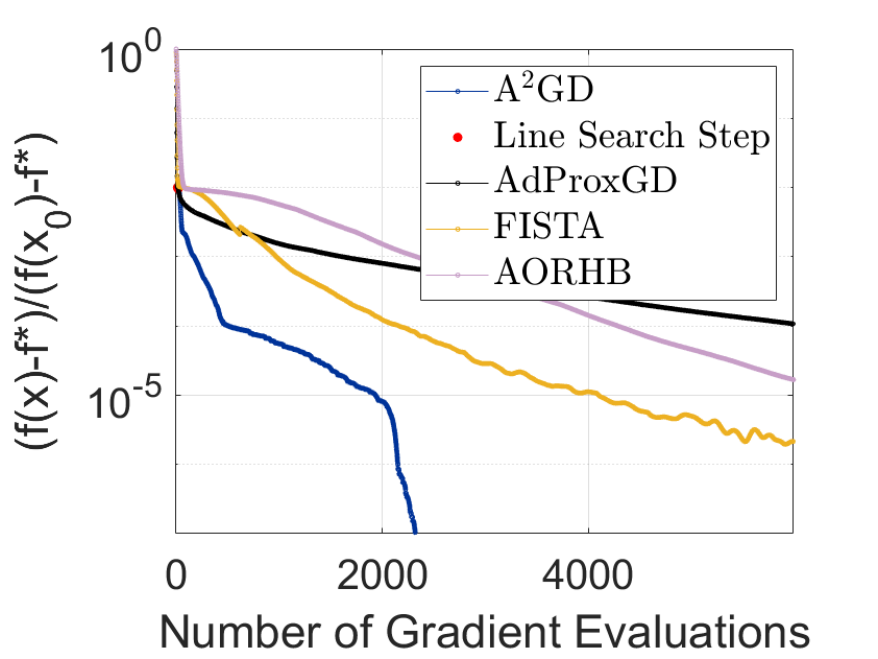}
    \captionsetup{hypcap=false}
    \captionof{figure}{Error curves under setting (2).
} \label{fig:mle_hard}
\end{minipage}

Following~\cite{Malitsky2023AdaptivePG}, we construct the data matrix $Y$ as follows: sample a random vector $y \in \mathbb{R}^n$, and define $y_i = y + \delta_i$ for $i = 1, \dots, M$, with $\delta_i \sim \mathcal{N}(0, I_n)$. Then set
$Y = \frac{1}{M} \sum_{i=1}^M y_i y_i^\top.$
We test our algorithm under two settings:  
(1) $n = 100$, $M = 50$, $\lambda_{\min} = 0.1$, $\lambda_{\max} = 10$;  
(2) $n = 50$, $M = 100$, $\lambda_{\min} = 0.1$, $\lambda_{\max} = 10^3$.

\begin{minipage}[t]{0.34\textwidth}
    \centering
        \label{fig:LASSO_500-1000d}
\includegraphics[width=0.99\linewidth]{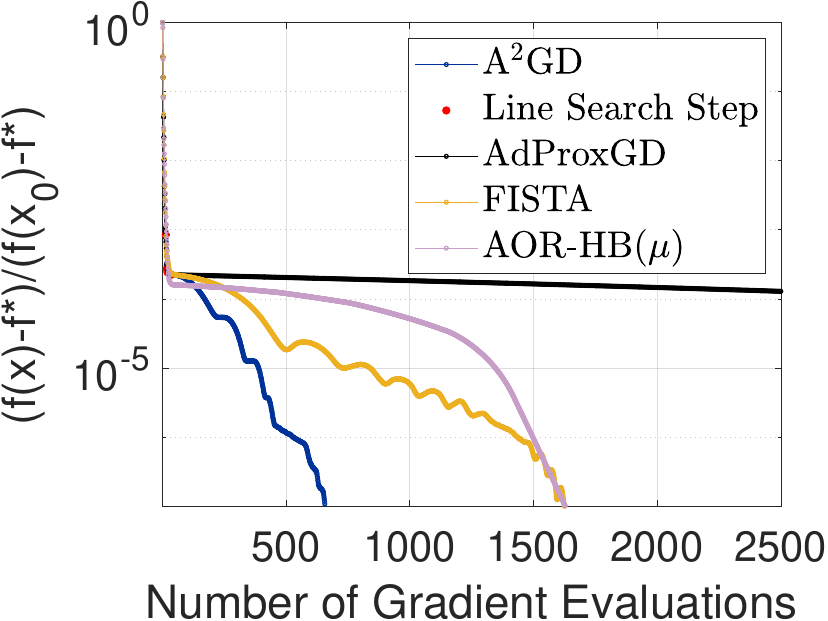}
\captionof{figure}{Error curve for $\ell_{1\text{-}2}$ problem with $n=500, p=1000$. 
%(\RV{figure 15 to be deleted}) 
%A$^2$GD (blue curve) converges much faster than AdProxGD (black curve) and FISTA (yellow curve).
}
\end{minipage}
\hfill
\begin{minipage}[t]{0.6\textwidth}
    \paragraph{$\ell_{1\text{-}2}$ nonconvex minimization problem} 
   We consider the $\ell_{1\text{-}2}$ minimization problem
\begin{equation}
    \min_{x\in\mathbb{R}^n} \quad\frac{1}{2}\|Ax-b\|^2+\lambda (\|x\|_1 - \|x\|_2),
\end{equation}
introduced by~\cite{Yin15_l1-2}, promotes sparser solutions than standard convex penalties. 
%In this example, we use the same choices of $A$, $b$, $\lambda$, and initialization as in the LASSO case above. Although the regularizer $g$ is non-convex, its proximal operator is available, and line search is applied solely to the smooth, convex term $h$; see~(\ref{eq:composite}).

%\vskip -8pt
%\begin{table}[htbp]
%    \centering
%    \captionof{table}{Results on $\ell_{1\text{-}2}$ problem.}
%    \label{tab:results_l12}
%    % \vskip -4pt
%    \resizebox{11.5cm}{!}{
%    \begin{tabular}{cc|cc|cc|cc|cc}
%        \toprule
%        \multicolumn{2}{c|}{Problem Size} & \multicolumn{2}{c|}{A$^2$GD} & \multicolumn{2}{c|}{AdProxGD} & \multicolumn{2}{c|}{FISTA} & \multicolumn{2}{c}{AOR-HB} \\
%        \cmidrule(lr){1-2} \cmidrule(lr){3-4} \cmidrule(lr){5-6} \cmidrule(lr){7-8} \cmidrule(lr){9-10}
%        $n$ & $p$ & \#Grad & Time & \#Grad & Time & \#Grad & Time & \#Grad & Time  \\
%        \midrule
%    % 6000 & 5000 & 182 & 14.88 & 414 & 38.37 & 497 & 42.98 & 231 & 24.75\\
%% 3000 & 5000 & 1045 & 39.70 & (>10000) & - & 4935 & 201.70 & 3434 & 144.02\\
%% 427.52
%500 & 1000 & 683 & 0.90 & 7828 & 12.80 & 1765 & 3.05 & 1585 & 2.37\\
%1000 & 2000 & 998 & 5.00 & (>10000) & - & 2725 & 15.16 & 2295 & 14.48\\
%        \bottomrule
%    \end{tabular}
%    }    
% \end{table}

The matrix $A \in \mathbb{R}^{n \times p}$ is generated from a standard Gaussian distribution, and the ground truth $x^* \in \mathbb{R}^p$ has sparsity $50$. The observation vector is constructed as $b = Ax^*$. We set the regularization parameter $\lambda = 1$ and the problem size: $n = 500$, $p = 1000$. The initial point is sampled as $x_0 = y_0 \sim 10\mathcal{N}(0, I_p)$.
\end{minipage}

\begin{comment}
\begin{minipage}[htdp]{0.4\textwidth}
    \centering
    \includegraphics[width=0.9\linewidth]{./graph/ablation/geval-l1-2.pdf}
    \captionsetup{hypcap=false}
    \captionof{figure}{Error curves in terms of gradient evaluations.}      \label{fig:l1-2-geval}
\end{minipage}%
\hfill
\begin{minipage}[htdp]{0.5\textwidth}
    \centering
    \includegraphics[width=0.7\linewidth]{./graph/ablation/exetime-l1-2.pdf}
    \captionsetup{hypcap=false}
    \captionof{figure}{Error curves in terms of execution time.} \label{fig:l1-2-time}
\end{minipage}
\RV{figures 16-17 to be replaced by figure 18.}

\begin{figure}[t]
    \centering

    \begin{subfigure}{0.46\textwidth}
        \centering
        \includegraphics[width=\linewidth]{graph/ablation/geval-l1-2.pdf}
        % \caption{First}
    \end{subfigure}
    \hfill
    \begin{subfigure}{0.46\textwidth}
        \centering
        \includegraphics[width=\linewidth]{graph/ablation/exetime-l1-2.pdf}
        % \caption{Second}
    \end{subfigure}

    \caption{Comparison of A$^2$GD with other baselines on the $\ell_{1-2}$ nonconvex problem. Results are in terms of number of gradients (left) and execution time (right).}
    \label{fig:l1-2_results}
\end{figure}
\end{comment}

\paragraph{\RV{Scaling behavior}.}
We consider the linear finite element method for the Poisson problem
\[
- \Delta u = b \ \text{in } \Omega,\qquad u=0 \ \text{on }\partial\Omega,
\]
where $\Omega$ is the unit disk discretized by a quasi-uniform triangulation $\mathcal{T}_h$. 

%
%\begin{table}[htdp]
%\centering
%\begin{tabular}{l|l|l|l|l}
%\hline
%$h$     & $n$     & $\kappa$ & \# Grad & Total Time \\ \hline
%1/20  & 1262  & 7.85e2   & 184          & 0.01       \\ \hline
%1/40  & 5167  & 3.15e3   & 303          & 0.07       \\ \hline
%1/80  & 20908 & 1.30e4   & 607          & 0.56       \\ \hline
%1/160 & 84120 & 5.32e4   & 987          & 4.52       \\ \hline
%\end{tabular}
%\caption{The scaling in number of gradient evaluations and total running time of A$^2$GD.}
%\label{tab:scaling_laplace2D}
%\end{table}

\begin{table}[htbp]
    \centering
    \caption{Performance comparison on 2D linear Laplacian problem, $\mathrm{tol}=10^{-6}$}
    \label{tab:Poissonresults}
    \resizebox{\textwidth}{!}{
    \begin{tabular}{ccc|cc|cc|cc|cc}
        \toprule
        \multicolumn{3}{c|}{Problem Size} & \multicolumn{2}{c|}{A$^2$GD} & 
        \multicolumn{2}{c|}{AdProxGD} & \multicolumn{2}{c|}{NAG} & 
        \multicolumn{2}{c}{AOR-HB} \\
        \cmidrule(lr){1-3} \cmidrule(lr){4-5} \cmidrule(lr){6-7} 
        \cmidrule(lr){8-9} \cmidrule(lr){10-11}
        $h$ & $n$ & $\kappa$ & \#Grad & Time & \#Grad & Time & \#Grad & Time & \#Grad & Time \\
        \midrule
        $1/20$  & 1262   & 7.85e+02 & 162  & 0.02 & 895   & 0.05 & 583   & 0.02 & 248  & 0.01 \\
        $1/40$  & 5166   & 3.15e+03 & 293  & 0.10 & 3191  & 0.90 & 1205  & 0.26 & 418  & 0.11 \\
        $1/80$  & 20908  & 1.30e+04 & 476  & 0.65 & 10729 & 15.72 & 1651  & 1.83 & 699  & 0.81 \\
        $1/160$ & 84120  & 5.32e+04 & 791  & 5.06 & (>20000) & 131.19 & 2902 & 14.51 & 1187 & 6.11 \\
        \bottomrule
    \end{tabular}
    }
\end{table}

\vskip -4pt
\begin{minipage}[t]{0.375\textwidth}
    \centering
        \label{fig:scaling_Laplace2D}
\includegraphics[width=0.99\linewidth]{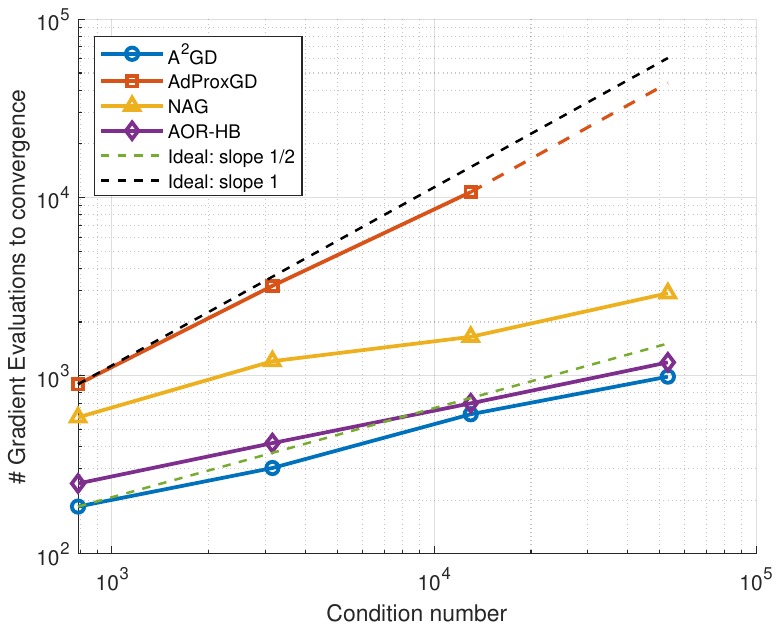}
\captionof{figure}{A$^{2}$GD, NAG, and AOR-HB exhibit the accelerated $\sqrt{\kappa}$ scaling, whereas AdProxGD follows the non-accelerated $\kappa$ rate.
}
\end{minipage}
\hfill
\begin{minipage}[t]{0.585\textwidth}
Using the \texttt{$i$FEM} package~\cite{chen:2008ifem}, we assemble the stiffness matrix $A$ and define the quadratic objective
\[
f(x)=\tfrac12 (x-x^*)^\top A(x-x^*),
\]
with $x^* \in \mathbb{R}^n$ and $x_0\sim \mathrm{Unif}(0,1)$ componentwise.

\smallskip
It is well known that $\kappa(A)=O(h^{-2})=O(n)$. We estimate $L$ and $\mu$ using the extreme eigenvalues of $A$. We use quasi-uniform meshes on the disk rather than structured square grids, where closed-form eigenvalue bounds are available and adaptive selection of $L$ and $\mu$ is less critical.

\smallskip
Table~\ref{tab:Poissonresults} shows that when the condition number increases by a factor of $4$, the number of gradient steps for accelerated methods grows by roughly a factor of $2$, and the total runtime by about a factor of $8$, since each gradient evaluation becomes $4$ times more expensive.
\end{minipage}
\smallskip

Figure~\ref{fig:scaling_Laplace2D} compares scaling behavior across methods: accelerated methods exhibit slopes near $1/2$, whereas the non-accelerated AdProxGD scales with slope close to $1$.

\vskip -4pt
\paragraph{\RV{Ablation study on the warm-up phase and hyper-parameters.}}
We include an adaptive gradient-descent warm-up in A$^2$GD to automatically select $\mu_0$ and $R$, although this is not essential, since the method degrades only moderately when these hyper-parameters are misspecified. In experiments on the regularized logistic regression problem, A$^2$GD remains robust even when these parameters vary by factors of $10^3$; see Fig.~\ref{fig:ablation_warmup}. 

\begin{figure}[htp]
    \centering

    \begin{subfigure}{0.43\textwidth}
        \centering
        \includegraphics[width=0.85\linewidth]{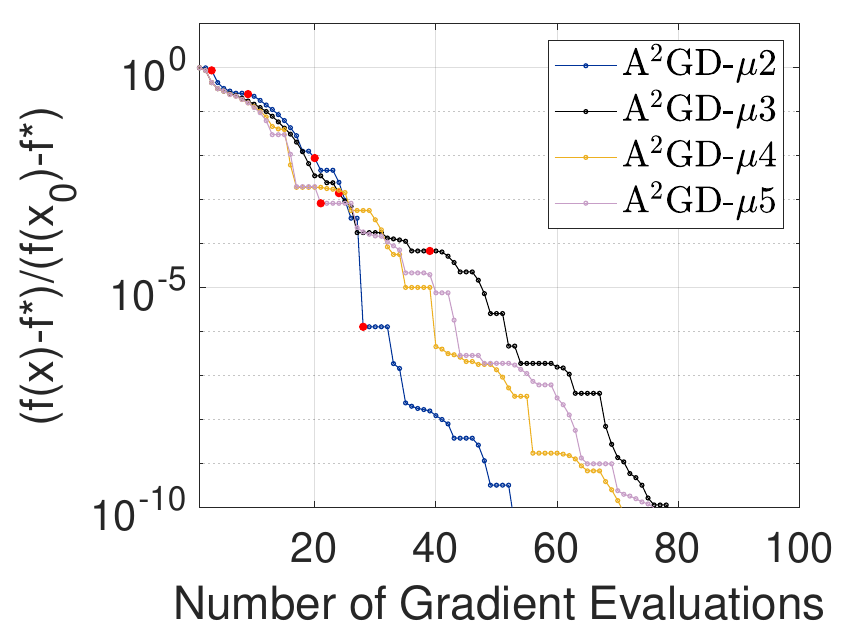}
        % \caption{First}
    \end{subfigure}
%    \hfill
%    \begin{subfigure}{0.23\textwidth}
%        \centering
%        \includegraphics[width=\linewidth]{graph/ablation/exetime_manualmu.pdf}
%        % \caption{Second}
%    \end{subfigure}
    \hfill
    \begin{subfigure}{0.43\textwidth}
        \centering
        \includegraphics[width=0.85\linewidth]{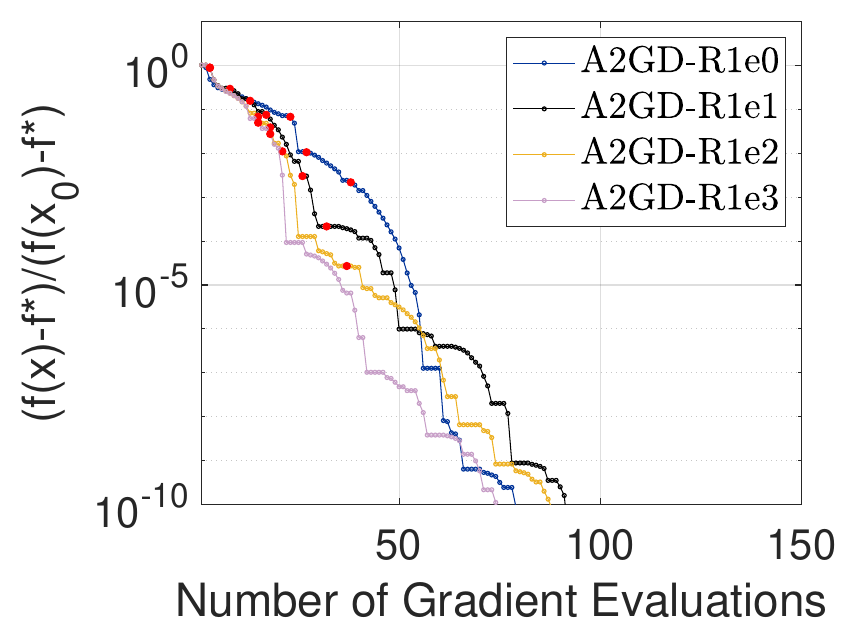}
        % \caption{Third}
    \end{subfigure}
%    \hfill
%    \begin{subfigure}{0.23\textwidth}
%        \centering
%        \includegraphics[width=\linewidth]{graph/ablation/exetime_manualR.pdf}
%        % \caption{Fourth}
%    \end{subfigure}
\caption{Comparison of A$^{2}$GD variants without warm-up using manual initializations. The left panel shows the effect of different choices of $\mu_0$, and the right panel shows the effect of varying $R$. For $2\le i\le 5$, “A$^{2}$GD-$\mu_i$’’ denotes the manual choice $\mu_0=10^{-i}$, with $i=2$ closest to the warm-up value; for $0\le j\le 3$, “A$^{2}$GD-R1e$j$’’ denotes the manual choice $R=10^j$, with $j=0$ matching the warm-up.}
    \label{fig:ablation_warmup}
\end{figure}

\RV{Across all tests, our A$^2$GD method consistently outperforms baseline algorithms. Empirically, line search is triggered only a few times, typically fewer than $10$, and almost all activations occur in the early phase. In the middle and late stages, line search are rare or completely absent across all tested problems.}

\bibliographystyle{plainnat}
\bibliography{Optimization.bib,reference}

% !TEX root =  AAGDneurips_2025.tex

\section*{Appendix A: Adaptive Gradient Descent Method}\label{appendix:AGD}

We present an algorithm for adaptive gradient descent method (Ad-GD) which is a simplified version of A$^2$GD without momentum. 

\begin{algorithm}[htbp]
\caption{Adaptive Gradient Descent Method (Ad-GD)}\label{alg:adaptiveGD2e}
\KwIn{Initial point $x_0 \in \mathbb{R}^n$, initial step size $L_0 > 0$, initial strong convexity estimate $\mu_0 > 0$}
\KwOut{Sequence $\{x_k\}$}

\For{$k = 0,1,2,\dots$}{
    $x_{k+1} \gets x_k - \frac{1}{L_k} \nabla f(x_k)$\;

    $b_k^{(1)} \gets \dfrac{1}{2L_k} \|\nabla f(x_{k+1}) - \nabla f(x_k)\|^2 - D_f(x_k, x_{k+1})$\;

    $b_k^{(2)} \gets -\dfrac{1}{2L_k} \|\nabla f(x_k)\|^2$\;

    $p_k \gets \left(1 + \frac{\mu_k}{L_k}\right)^{-1} \left(p_{k-1} + b_k^{(1)} + b_k^{(2)}\right)$\;

    \If{$p_k > 0$}{
        Use adaptive backtracking to update $L_k$ so that $b_k^{(1)} \leq 0$\;
    }

    $L_k \gets \dfrac{\|\nabla f(x_k) - \nabla f(x_{k+1})\|^2}{2 D_f(x_k, x_{k+1})}$\;

    $\mu_k \gets \min\{\mu_k, L_k\}$\;
}
\end{algorithm}

There are several variants of Ad-GD depending on how we define $\delta_k$ and split $b_k^{(1)}$ and $b_k^{(2)}$. For example, we can use $\delta_k = 1- \mu/L_k$, $b_k^{(2)} = 0$, and 
\begin{equation}
 b_k^{(1)} :=  \frac{1}{2L_k}\|\gf(x_{k+1})-\gf(x_k)\|^2 - D_f(x_k,x_{k+1})
           - \frac{1}{2L_k}\|\gf(x_{k+1})\|^2. 
\end{equation}
The inequality $b_k^{(1)} \leq 0 $ is equivalent to the criteria proposed by Nesterov in 
 \cite{Nesterov2012GradientMF}.
\begin{proposition}
The inequality $b_k^{(1)} \leq 0 $ is equivalent to
\begin{equation}\label{nestorov_stopping}
        m_{L_k}(x_{k+1};x_k)\geq f(x_{k+1}),
    \end{equation}
    where $m_{L_k}(y;x)=f(x)+\dual{\gf(x),y-x}+\frac{L_k}{2}\|y-x\|^2$, $x_{k+1}=\argmin_{y} m_{L_k}(y;x_k)$.  
\end{proposition}
%    For the equivalence proof, we refer to Appendix A. 
    
\begin{proof}
 First, we have the identity
 \begin{align*}\label{eq:simple_identity}
     &f(x_k)-f(x_{k+1})-\frac{\alpha_k}{2}\|\gf(x_k)\|^2\\={}&\frac{\alpha_k}{2}\|\gf(x_{k+1})\|^2-\frac{\alpha_k}{2}\|\gf(x_{k+1})-\gf(x_k)\|^2+D_f(x_k,x_{k+1}).
 \end{align*}
 Notice that    $ x_{k+1}=\argmin_{y} m_{L_k}(y;x_k) \Leftrightarrow x_{k+1}=x_k-\frac{1}{L_k}\gf(x_k), $
 so
 \begin{align*}
     &m_{L_k}(x_{k+1};x_k)\geq f(x_{k+1})\\\Leftrightarrow & f(x_k)+\dual{\gf(x_k),-\frac{1}{L_k}\gf(x_k)}+\frac{L_k}{2}\left\|\frac{1}{L}\gf(x_k)\right\|^2\geq f(x_{k+1})\\
     \Leftrightarrow & f(x_k)-f(x_{k+1})-\frac{1}{2L_k}\|\gf(x_k)\|^2\geq 0\\
     \Leftrightarrow & \frac{1}{2L_k}\|\gf(x_{k+1})\|^2+\frac{1}{2L_k}\|\gf(x_{k+1})-\gf(x_k)\|^2-D_f(x_k,x_{k+1})\geq 0.
 \end{align*}
 Thus, equivalence is proved.
 \end{proof}
% !TEX root =  AAGDneurips_2025.tex

\section*{Appendix B: Identities of Accelerated Gradient Methods} \label{sec:B}
We will use the Hessian-based Nesterov accelerated gradient (HNAG) flow proposed in \cite{chen2019orderoptimizationmethodsbased}
\begin{equation}% \label{eq:Hagf-intro}
	\left\{
	\begin{aligned}
		x' = {}&y-x-\beta\nabla f(x),\\
		y'={}&x - y -\frac{1}{\mu}\nabla f(x).
	\end{aligned}
	\right.
\end{equation}
Denote by $\bs z=(x, y)^{\intercal}$ and $\mathcal G(\bs z)$ the right hand side of (\ref{eq:Hagf-intro}), which now becomes $\bs z' = \mathcal G(\bs z)$. In the notation $\nabla \mathcal E$, we consider $\mu$ as a fixed parameter and take derivative with respect to $\bs z$. 

\begin{lemma}
We have the identity
 \begin{equation}
 \begin{aligned}
	-\nabla \mathcal E(\bs z) \cdot \mathcal G(\bs z) = {}&
\mathcal E(\bs z)+ \beta \nm{\nabla f(x)}_*^2  +\frac{\mu}{2}\nm{y-x}^2 +  D_f(x^*, x) -\frac{\mu}{2}\nm{x - x^{\star}}^2.	
\end{aligned}
\end{equation}
\end{lemma}
\begin{proof}
A direct computation gives 
\begin{equation}\label{eq:A-HNAG}
\begin{split}
	&-\nabla \mathcal E(\bs z) \cdot \mathcal G(\bs z) 
= 
\begin{pmatrix}
\nabla f(x)\\
 \mu (y-x^{\star})
\end{pmatrix}
\begin{pmatrix}
(x-x^{\star}) - (y - x^{\star})+\beta\nabla f(x)\\ 
(y - x^{\star})- (x-x^{\star}) +\frac{1}{\mu}\nabla f(x)\\ 
\end{pmatrix}
	\\
	={}& \dual{\nabla f(x),x-x^{\star}} + \beta\nm{\nabla f(x)}_*^2+ \mu\nm{y-x^{\star}}^2 -\mu (y-x^{\star}, x - x^{\star})\\
=	{}&  \mathcal E(\bs z)+ \beta \nm{\nabla f(x)}_*^2  +  D_f(x^*, x)  +\frac{\mu}{2}\nm{y-x}^2 -\frac{\mu}{2}\nm{x - x^{\star}}^2.	
\end{split}
\end{equation}
\end{proof}

\begin{lemma}\label{lem:identityA2GD}
We have the identity
 	\begin{equation*}
		%	\label{diff-Lk}
		\begin{split}
&(1+\alpha_k) \mathcal E(\bs z_{k+1}; \mu_{k}) - \mathcal E(\bs z_k; \mu_k) 
			\\
= ({\rm I}) &\quad \frac{1}{2}\left ( \frac{\alpha_k^2}{\mu_k}  - \frac{1}{L_k}\right )\nm{\nabla f(x_{k+1})}_*^2 \\
({\rm II})	& +\frac{1}{2L_k}\| \nabla f(x_{k+1}) - \nabla f(x_k) \|^2 - D_f(x_{k}, x_{k+1})\\
({\rm III})			& -\frac{1}{2L_k}\nm{\nabla f(x_k)}_*^2+ \frac{\alpha_k\mu_k}{2}\left (\nm{x_{k+1}-x^{\star}}^2 - \frac{2}{\mu_k} D_f(x^{\star}, x_{k+1}) - (1+\alpha_k)\nm{x_{k+1}-y_{k+1}}^2\right ) .
		\end{split}
	\end{equation*}
\end{lemma}
\begin{proof}
Treat $\mu_k$ as a fixed parameter. We expand the difference
\begin{equation}% \label{eq:Elambda}
\mathcal E(\bs z_{k+1}; \mu_{k}) - \mathcal E(\bs z_k; \mu_k) = \langle \nabla \mathcal E(\bs z_{k+1}; \mu_{k}), \bs z_{k+1} - \bs z_k \rangle - D_{\mathcal E}(\bs z_k, \bs z_{k+1};\mu_k),
\end{equation}
where the negative term $- D_{\mathcal E}(\bs z_k, \bs z_{k+1};\mu_k)$ is expanded as $- D_f(x_k, x_{k+1}) - \frac{\mu_k}{2}\| y_{k} - y_{k+1}\|^2.$

Using the identity~(\ref{eq:A-HNAG}) in the continuous level, we have
$$
\begin{aligned}
&\langle \nabla \mathcal E(\bs z_{k+1}; \mu_{k}), \alpha_k \mathcal G(\bs z_{k+1}, \mu_k) \rangle = - \alpha_k \mathcal E(\bs z_{k+1}, \mu_k)\\
&  - \frac{1}{L_k} \nm{\nabla f(x_{k+1})}_*^2  - \alpha_k D_f(x^*, x_{k+1}) + \frac{\alpha_k\mu_k} {2}\left (\nm{x_{k+1}-x^{\star}}^2 - \nm{x_{k+1}-y_{k+1}}^2\right ).	
\end{aligned}
$$

The difference between the scheme and the implicit Euler method is
$$
\bs z_{k+1} - \bs z_k - \alpha_k \mathcal G(\bs z_{k+1}, \mu_k) = \alpha_k 
\begin{pmatrix}
 y_k - y_{k+1} + \beta_k (\nabla f(x_{k+1}) - \nabla f(x_{k}))\\
0
\end{pmatrix}.
$$
which will bring more terms
	\begin{align*}
&		\dual{\nabla_x \mathcal E(\bs z_{k+1}, \mu_{k}), \bs z_{k+1} - \bs z_k - \alpha_k \mathcal G(\bs z_{k+1}, \mu_k)} \\
		&= \frac{1}{L_k} \left ( \nabla f(x_{k+1}) , \nabla f(x_{k+1}) - \nabla f(x_k)\right )+\alpha_k \dual{\nabla f(x_{k+1}), y_k - y_{k+1}}.
	\end{align*}
	We then use the identity of squares for the cross term of gradients
	\begin{align*}
		&\frac{1}{L_k}( \nabla f(x_{k+1}) , \nabla f(x_{k+1}) - \nabla f(x_k)) \\
		= & - \frac{1}{2L_k}\| \nabla f(x_{k})\|_*^2
		+ \frac{1}{2L_k}\| \nabla f(x_{k+1})\|_*^2 + \frac{1}{2L_k}\| \nabla f(x_{k+1}) - \nabla f(x_k) \|_*^2.
	\end{align*}
As expected, this cross term brings more positive squares but also contribute a negative one.

On the second term, we write as
$$
\begin{aligned}
&\alpha_k \dual{\nabla f(x_{k+1}), y_k - y_{k+1}} ={} \dual{\frac{\alpha_k}{\sqrt{\mu_k}}\nabla f(x_{k+1}), \sqrt{\mu_k}(y_k - y_{k+1})}\\
={}& \frac{\alpha_k^2}{2\mu_k} \| \nabla f(x_{k+1})\|_*^2 + \frac{\mu_k}{2}\nm{y_k - y_{k+1}}^2 - \frac{1}{2}\nm{ \frac{\alpha_k}{\sqrt{\mu_k}}\nabla f(x_{k+1}) - \sqrt{\mu_k} (y_{k} - y_{k+1})}^2\\
={}& \frac{\alpha_k^2}{2\mu_k} \| \nabla f(x_{k+1})\|_*^2 + \frac{\mu_k}{2}\nm{y_k - y_{k+1}}^2 - \frac{1}{2}\alpha_k^2\mu_k\nm{ x_{k+1} - y_{k+1}}^2.
\end{aligned}
$$
Combining altogether, we get the desired identity. 
\end{proof}

\subsection*{Proof of Theorem \ref{thm:conv-ex1-ode-NAG}}
First, we prove convergence of Algorithm \ref{alg:adaptiveHNAG} within a single inner iteration, i.e. $\varepsilon$ is fixed, in the following lemma. It bears similarity to \citep[Theorem 8.3]{chen2025acceleratedgradientmethodsvariable}, and is a direct result of Lemma \ref{lem:identityA2GD}.

\begin{lemma}
\label{lem:single_inner_iter_conv}
Suppose $f$ is convex and $L$-smooth. Let $z_k=(x_k,y_k)$ be the iterates generated by Algorithm~\ref{alg:adaptiveHNAG} within an inner iteration where $\mu=\varepsilon$. Assume that there exists $R>0$ such that
\[
\|x_k - x^*\| \leq R, \quad \forall\, k\geq 0,
\]
and that there exists $l \in (\varepsilon,L)$ such that $L_k \geq l$ for all $k \geq 0$. Then the Lyapunov function exhibits linear convergence up to a perturbation:
\[
\mathcal{E}(z_k;\varepsilon)
\;\leq\;
\left(\frac{1}{1+\sqrt{\varepsilon/(rL)}}\right)^k \mathcal{E}(z_0;\varepsilon)
\;+\;\frac{\varepsilon}{2} R^2,
\]
where $r$ is the backtracking ratio (in Algorithm~\ref{alg:adaptiveHNAG}, $r=3$).
\end{lemma}

\begin{proof}
By Lemma~\ref{lem:identityA2GD}, we have
\[
\mathcal{E}(z_{k+1};\mu_{k+1})
\;\leq\;
\frac{1}{1+\alpha_k}\mathcal{E}(z_k;\mu_k)
+ \frac{1}{1+\alpha_k}\bigl(b_k^{(1)}+b_k^{(2)}\bigr)
+ \frac{\alpha_k\mu_k}{2(1+\alpha_k)} R^2.
\]

Since $l \leq L_k \leq rL$, it follows that
\[
\sqrt{\tfrac{\varepsilon}{rL}}
\;\leq\; \alpha_k
\;\leq\; \sqrt{\tfrac{\varepsilon}{l}}.
\]
Therefore,
\[
\mathcal{E}(z_{k+1};\mu_{k+1})
\;\leq\;
\frac{1}{1+\sqrt{\varepsilon/(rL)}} \mathcal{E}(z_k;\mu_k)
+ \frac{1}{1+\alpha_k}\bigl(b_k^{(1)}+b_k^{(2)}\bigr)
+ \frac{\varepsilon \sqrt{\varepsilon/l}}{2\bigl(1+\sqrt{\varepsilon/l}\bigr)} R^2.
\]

Iterating the inequality yields
\[
\mathcal{E}(z_{k+1})
\;\leq\;
\left(\frac{1}{1+\sqrt{\varepsilon/(rL)}}\right)^{k+1}
\mathcal{E}(z_0)
+ p_{k+1}
+ \frac{\varepsilon \sqrt{\varepsilon/l}}{2\bigl(1+\sqrt{\varepsilon/l}\bigr)}
\sum_{i=0}^{k}
\left(\frac{1}{1+\sqrt{\varepsilon/l}}\right)^i R^2,
\]
where $p_{k+1}$ is the accumulated perturbation. By Algorithm~\ref{alg:adaptiveHNAG}, we have $p_{k+1} \leq 0$.

Finally, the geometric sum is bounded as
\[
\sum_{i=0}^k \left(\frac{1}{1+\sqrt{\varepsilon/l}}\right)^i
\;\leq\;
\frac{1+\sqrt{\varepsilon/l}}{\sqrt{\varepsilon/l}}.
\]
Substituting this estimate gives the claimed bound
\[
\mathcal{E}(z_{k+1};\varepsilon)
\;\leq\;
\left(\frac{1}{1+\sqrt{\varepsilon/(rL)}}\right)^{k+1} \mathcal{E}(z_0;\varepsilon)
+ \frac{\varepsilon}{2}R^2.
\]
\end{proof}

\begin{proof}[Proof of Theorem \ref{thm:conv-ex1-ode-NAG}]
We distinguish between the convex case ($\mu=0$) and the strongly convex case ($\mu>0$).

If $\mu=0$, in this case, the proof of \citep[Theorem~8.4]{chen2025acceleratedgradientmethodsvariable} applies directly, once the single-inner-iteration convergence relation (Lemma~\ref{lem:single_inner_iter_conv}) is established. Therefore, no further argument is needed.

If instead, $\mu>0$, recall that in the algorithm the effective radius is updated as
\[
R_k^2 \;=\; \Bigl(1-\tfrac{\mu}{\mu_k}\Bigr)R^2 .
\]
Thus, whenever $\mu_k \geq \mu$, we obtain $R_k^2 \leq 0$, which implies that further reduction of $\mu_k$ is no longer admissible. In particular, $\mu_k$ will stop decreasing once the tolerance parameter $\varepsilon$ satisfies $\varepsilon \leq \mu$.

Since $\varepsilon$ is halved at each outer stage, the final value of $\mu_k$ is therefore bounded below by $\mu/2$. At the same time, the smoothness parameter satisfies $L_k \leq rL$ by construction. Hence, in the terminal stage we obtain an effective condition number bounded by
\[
\kappa_{\mathrm{eff}}
= \frac{L_k}{\mu_k}
\;\leq\; \frac{rL}{\mu/2}
= \frac{2rL}{\mu}.
\]

Applying the convergence estimate from Lemma~\ref{lem:single_inner_iter_conv} in this regime, the Lyapunov function contracts linearly:
\[
\mathcal{E}_{k_s}
\;\leq\;
\Biggl(\frac{1}{1+\sqrt{\mu_k/L_k}}\Biggr)^{k_s}\, \mathcal{E}_0
\;\leq\;
\Biggl(\frac{1}{1+\sqrt{\mu/2rL}}\Biggr)^{k_s}\, \mathcal{E}_0.
\]

Therefore, to ensure that $\mathcal{E}_{k_s}\leq \mathrm{tol}\cdot \mathcal{E}_0$, it suffices to take
\[
k_s \;\geq\;
\frac{\ln(1/\mathrm{tol})}{\ln\!\left(1+\sqrt{\mu/2rL}\right)}
\;=\;\mathcal{O}\!\Bigl(\sqrt{2rL/\mu}\,\ln(1/\mathrm{tol})\Bigr).
\]

This establishes the desired complexity bound in both cases.
\end{proof}

\section*{Appendix C: Composite Convex Optimization}
We derive the continuous time analogy to Lemma 3.1. First, define the composite right hand side update
\begin{equation*}
    \mathcal{G}(z)=\left(y-x-\beta(\gh(x)+q),x-y-\frac{1}{\mu}(\gh(x)+q)\right)^{\mathrm T},
\end{equation*}
where $q\in\partial g(x)$. Let $\mathcal E_h(z;\mu)=h(x)-h(x^*)+\frac{\mu}{2}\|y-x^*\|^2$, then $\mathcal E(z;\mu)=\mathcal E_h(z;\mu)+(g(x)-g(x^*))$ is splitted into a smooth part and a non-smooth part.
\begin{lemma}\label{lem:continuousODEcomposite}
    We have the following inequality
    \begin{equation*}
        -\dual{\nabla\mathcal{E}_h(x)+\binom{q}{0},\mathcal{G}(z)}\geq \mathcal{E}(z)+\beta\|\gh(x)+q\|_*^2+\frac{\mu}{2}\|y-x\|^2+D_h(x^*,x)-\frac{\mu}{2}\|x-x^*\|^2.
    \end{equation*}
\end{lemma}

\begin{proof}
A direct computation gives 
\begin{equation}
\begin{split}
	&-\dual{\nabla\mathcal{E}_h(x)+\binom{q}{0},\mathcal{G}(z)}
= 
\begin{pmatrix}
\nabla h(x)+q\\
 \mu (y-x^{\star})
\end{pmatrix}
\begin{pmatrix}
(x-x^{\star}) - (y - x^{\star})+\beta(\nabla h(x)+q)\\ 
(y - x^{\star})- (x-x^{\star}) +\frac{1}{\mu}(\nabla h(x)+q)\\ 
\end{pmatrix}
	\\
	={}& \dual{\nabla h(x)+q,x-x^{\star}} + \beta\nm{\nabla h(x)+q}_*^2+ \mu\nm{y-x^{\star}}^2 -\mu (y-x^{\star}, x - x^{\star})\\
\geq	{}&  \mathcal E(\bs z)+ \beta \nm{\nabla h(x)+q}_*^2  +  D_h(x^*, x)  +\frac{\mu}{2}\nm{y-x}^2 -\frac{\mu}{2}\nm{x - x^{\star}}^2,
\end{split}
\end{equation}
the last inequality following from $q\in\partial g(x)$.
\end{proof}

\begin{lemma}
    We have the following inequality
 	\begin{equation*}
		%	\label{diff-Lk}
		\begin{split}
&(1+\alpha_k) \mathcal E(\bs z_{k+1}; \mu_{k}) - \mathcal E(\bs z_k; \mu_k) 
			\\
\leq ({\rm I}) &\quad \frac{1}{2}\left ( \frac{\alpha_k^2}{\mu_k}  - \frac{1}{L_k}\right )\nm{\nabla h(x_{k+1}) + q_{k+1} }_*^2 \\
({\rm II})	& +\frac{1}{2L_k}\| \nabla h(x_{k+1}) - \nabla h(x_k) \|^2 - D_h(x_{k}, x_{k+1})\\
({\rm III})			& -\frac{1}{2L_k}\nm{\nabla h(x_k) + q_{k+1} }_*^2+ \frac{\alpha_k\mu_k}{2}\left (\nm{x_{k+1}-x^{\star}}^2 - \frac{2}{\mu_k} D_h(x^{\star}, x_{k+1}) - (1+\alpha_k)\nm{x_{k+1}-y_{k+1}}^2\right ) .
		\end{split}
	\end{equation*}
\end{lemma}

\begin{proof}
The proof is similar to the smooth convex case. Expand the difference of $\mathcal{E}$ at $z_{k+1}$,
\begin{equation}\label{eq:Elambda}
\mathcal E(\bs z_{k+1}; \mu_{k}) - \mathcal E(\bs z_k; \mu_k) \leq \langle \nabla \mathcal E_h(\bs z_{k+1}; \mu_{k}) + \binom{q_{k+1}}{0}, \bs z_{k+1} - \bs z_k \rangle - D_{\mathcal E_h}(\bs z_k, \bs z_{k+1};\mu_k),
\end{equation}
where the negative term $- D_{\mathcal E_h}(\bs z_k, \bs z_{k+1};\mu_k)$ is expanded as $- D_h(x_k, x_{k+1}) - \frac{\mu_k}{2}\| y_{k} - y_{k+1}\|^2.$ The inequality is due to the definition of the subgradient.

From Lemma \ref{lem:continuousODEcomposite}, we have
$$
\begin{aligned}
&\langle \nabla \mathcal E_h(\bs z_{k+1}; \mu_{k})+\binom{q_{k+1}}{0}, \alpha_k \mathcal G(\bs z_{k+1}, \mu_k) \rangle \leq - \alpha_k \mathcal E(\bs z_{k+1}, \mu_k)\\
&  - \frac{1}{L_k} \nm{\nabla h(x_{k+1})+q_{k+1}}_*^2  - \alpha_k D_h(x^*, x_{k+1}) + \frac{\alpha_k\mu_k} {2}\left (\nm{x_{k+1}-x^{\star}}^2 - \nm{x_{k+1}-y_{k+1}}^2\right ).	
\end{aligned}
$$

The difference between the scheme and the implicit Euler method is
$$
\bs z_{k+1} - \bs z_k - \alpha_k \mathcal G(\bs z_{k+1}, \mu_k) = \alpha_k 
\begin{pmatrix}
 y_k - y_{k+1} + \beta_k (\nabla h(x_{k+1}) - \nabla h(x_{k}))\\
0
\end{pmatrix}.
$$
which will bring more terms
	\begin{align*}
&		\dual{\nabla_x \mathcal E_h(\bs z_{k+1}, \mu_{k}) + q_{k+1}, \bs z_{k+1} - \bs z_k - \alpha_k \mathcal G(\bs z_{k+1}, \mu_k)} \\
		&= \frac{1}{L_k} \left ( \nabla h(x_{k+1}) + q_{k+1} , \nabla h(x_{k+1}) - \nabla h(x_k)\right )+\alpha_k \dual{\nabla h(x_{k+1}) + q_{k+1}, y_k - y_{k+1}}.
	\end{align*}
    
For the first term, we use the identity of squares % for the cross term of gradients
	\begin{align*}
		&\frac{1}{L_k}( \nabla h(x_{k+1}) + q_{k+
        1} , \nabla h(x_{k+1}) - \nabla h(x_k)) \\
		= & - \frac{1}{2L_k}\| \nabla h(x_{k})+q_{k+1}\|_*^2
		+ \frac{1}{2L_k}\| \nabla h(x_{k+1})+q_{k+1}\|_*^2 + \frac{1}{2L_k}\| \nabla h(x_{k+1}) - \nabla h(x_k) \|_*^2.
	\end{align*}
As expected, this cross term brings more positive squares but also contribute a negative one.

For the second term, we rewrite as
$$
\begin{aligned}
&\alpha_k \dual{\nabla h(x_{k+1})+q_{k+1}, y_k - y_{k+1}} ={} \dual{\frac{\alpha_k}{\sqrt{\mu_k}}\nabla h(x_{k+1})+q_{k+1}, \sqrt{\mu_k}(y_k - y_{k+1})}\\
={}& \frac{\alpha_k^2}{2\mu_k} \| \nabla h(x_{k+1})+q_{k+1}\|_*^2 + \frac{\mu_k}{2}\nm{y_k - y_{k+1}}^2 - \frac{1}{2}\nm{ \frac{\alpha_k}{\sqrt{\mu_k}}( \nabla h(x_{k+1}) + q_{k+1}) - \sqrt{\mu_k} (y_{k} - y_{k+1})}^2\\
={}& \frac{\alpha_k^2}{2\mu_k} \| \nabla h(x_{k+1})+q_{k+1}\|_*^2 + \frac{\mu_k}{2}\nm{y_k - y_{k+1}}^2 - \frac{1}{2}\alpha_k^2\mu_k\nm{ x_{k+1} - y_{k+1}}^2.
\end{aligned}
$$
Combining altogether, we get the desired inequality. 
\end{proof}

\begin{algorithm}[htbp]
\caption{A$^2$GD method for composite optimization}
\label{alg:adaptiveHNAGcomposite}

\KwIn{$x_0, y_0 \in \mathbb{R}^n$, $L_0, \mu_0, R > 0$, $\mathrm{tol}>0$, $\varepsilon>0$, $m\geq 1$}
% \For{$k = 0, 1, 2, \dots$}{
\While{$k=0$ or $\|\nabla f(x_k)+q_k\|> {\rm tol}\|\nabla f(x_0)\|$}{
    $\alpha_k \gets \sqrt{\mu_k / L_k}$\;
    $w_{k+1} \gets \frac{1}{\alpha_k+1}x_k + \frac{\alpha_k}{\alpha_k+1}y_k - \frac{1}{L_k(\alpha_k+1)}\gh(x_k)$\;
    $x_{k+1} \gets \mathrm{prox}_{\frac{1}{L_k(\alpha_k+1)}g}(w_{k+1})$\;
    $q_{k+1} \gets L_k(\alpha_k+1)(w_{k+1} - x_{k+1})$\;
    $y_{k+1} \gets \frac{\alpha_k}{\alpha_k+1}x_{k+1} + \frac{1}{\alpha_k+1}y_k - \frac{\alpha_k}{\mu_k(\alpha_k+1)}(\gh(x_{k+1}) + q_{k+1})$\;

    $b_k^{(1)} \gets \frac{1}{2L_k} \|\gh(x_{k+1}) - \gh(x_k)\|^2 - D_h(x_k, x_{k+1})$\;
    $b_k^{(2)} \gets -\frac{1}{2L_k} \|\nabla h(x_k) + q_{k+1}\|_*^2 + \frac{\alpha_k \mu_k}{2} \left(R^2 - (1+\alpha_k)\|x_{k+1} - y_{k+1}\|^2 \right)$\;
    $p_k \gets \frac{1}{1+\alpha_k}(p_{k-1} + b_k^{(1)} + b_k^{(2)})$\;

    \If{$p_k > 0$}{
        % \While{$b_k^{(1)} > 0$ \textbf{or} $b_k^{(2)} > 0$}{
            \If{$b_k^{(1)} > 0$}{
            %     Update $L_k$ via adaptive backtracking so that $b_k^{(1)} \leq 0$\;
            $v\gets \frac{2L_k D_f(x_k,x_{k+1})}{\|\nabla f(x_{k+1}) - \nabla f(x_k)\|^2}$,~$L_k\gets 3L_k/v$\;
            }
            \If{$b_k^{(2)} > 0$}{ 
                $\mu_k \gets \max\left\{\varepsilon,\min\left\{\mu_k, \frac{\|\gh(x_k)+q_{k+1}\|^{4/3}}{L_k^{1/3}(R^2 - (1+\alpha_k)\|x_{k+1} - y_{k+1}\|^2)^{2/3}}\right\}\right\}$\;
            }
            Go to line 2\;
        % }
    }

    % \If{$b_k^{(2)} > 0$}{
    \Else{
        $L_k \gets \frac{\|\gh(x_{k+1}) - \gh(x_{k})\|^2}{2 D_h(x_k, x_{k+1})}$\;
$\mu_{k+1} \gets \max\left\{\varepsilon,\min\left\{\mu_k, \frac{\|\gh(x_k)+q_{k+1}\|^{4/3}}{L_k^{1/3}(R^2 - (1+\alpha_k)\|x_{k+1} - y_{k+1}\|^2)^{2/3}}\right\}\right\}$\;
    
    }
\If{\textbf{decay condition}}{$\varepsilon\gets \varepsilon/2$\;
$m \gets \lfloor \sqrt{2} \cdot m \rfloor + 1$\;}
            $k\gets k+1$\;
}
\end{algorithm}

\begin{theorem}\label{thm:conv-ex1-ode-NAG-composite}
Let $(x_k, y_k)$ be the iterates generated by Algorithm \ref{alg:adaptiveHNAGcomposite}. Assume function $f$ is $\mu$-convex with $\mu\geq 0$. Assume there exists $R>0$ such that 
$$\|x_k-x^*\|\leq R,\qquad \forall~k\geq 0.$$

Let $k_s$ be the total number of steps after halving $\varepsilon$ exactly $s$ times, i.e. $\varepsilon = 2^{-s}\varepsilon_0$. 

\setlist[enumerate]{leftmargin=20pt,labelsep=0.6em}

\begin{enumerate}

\item When $\mu=0$, ther exists a constant $C > 0$ so that
$$
\frac{\mathcal{E}_{k_s}}{\mathcal{E}_0} \leq \frac{R^2 + 1}{\left( C k_s + \varepsilon_0^{-1/2} \right)^2} = \mathcal{O}\left( \frac{1}{k_s^2} \right)
$$

\item When $\mu > 0$, the iteration number to achieve $\mathcal{E}_{k_s}/\mathcal{E}_0 \leq (R^2+1) 2^{-s}\varepsilon_0\leq {\rm tol}$ is bounded by $\mathcal O(\sqrt{L/\mu}\ \ln {\rm tol})$, 

\end{enumerate}
where $\mathcal E_k = \mathcal E(\bs z_k; \mu_k)= {}f(x_k)-f(x^{\star}) + \frac{\mu_k}{2} \nm{y_k-x^{\star}}^2.$

\end{theorem}
\section*{\RV{Appendix D: Discussion}}\label{appendix:Ablation}

% \textcolor{red}{all except the study on $\varepsilon$ to be deleted}
\paragraph{The Relationship of Running Time with Gradient Evaluations}

In the main text, we report convergence primarily in terms of gradient evaluations, which dominate the computational cost for all methods. Consequently, wall-clock time is essentially proportional to the number of gradient (or proximal-gradient) evaluations. To confirm this, we record detailed timings on the composite MLE task, counting each proximal step as one gradient evaluation.

\begin{table}[h]
\centering
\begin{tabular}{l c c c c c c}
\toprule
Method & \# Iter & \# Grad Eval & Total time & Time/Iter & Grad Time & Grad \% \\
\midrule
A$^2$GD     & 2376  & 2382  & 2.71  & $1.14\times 10^{-3}$ & 1.71 & 63.2\% \\
AdProxGD & 20941 & 20941 & 18.68 & $8.92\times 10^{-4}$ & 14.68 & 78.6\% \\
FISTA    & 18041 & 18041 & 18.21 & $1.01\times 10^{-3}$ & 13.80 & 75.8\% \\
AOR-HB   & 8877  & 8877  & 8.81  & $9.93\times 10^{-4}$ & 6.70 & 76.1\% \\
\bottomrule
\end{tabular}
\caption{Computation cost breakdown on the MLE problem (2).}
\end{table}

Gradient (and proximal-gradient) evaluations account for over 60\% of the total running time for every method. Although A$^2$GD incurs slightly higher per-iteration cost due to a few extra vector operations, its much smaller number of gradient evaluations yields nearly a 70\% reduction in total time. We also provide error curves versus wall-clock time for the regularized logistic regression on Adult Census Income.

\paragraph{Ablation Study on the Choice of Hyper-parameter $\varepsilon$}
The tolerance $\varepsilon$ is a small positive number that controls $\mu_k$ from below. It is essential in the proof of linear/sub-linear convergence (Theorem \ref{thm:conv-ex1-ode-NAG}). Numerically, manipulating this parameter will keep the convergence of the algorithm, while making a difference to the convergence rate. This is verified in an ablation study on the regularized logistic regression problem on the Adult Census Income dataset. For $-2\leq i\leq 2$, the algorithm named "A2GD-eps$i$" means the manual choice $\varepsilon=1/10^{-6-i}$, where $i=0$ gives the choice of $\varepsilon$ that we use for all numerical examples in the main context. Figure \ref{fig:ablation_manualeps} agrees with the theory, and shows robustness of A2GD with the parameter $\varepsilon$.

% \begin{figure}
%     \centering
%     \includegraphics[width=0.5\linewidth]{graph/ablation/geval_manualeps.pdf}
%     \caption{A2GD with different choices of $\varepsilon$, in terms of gradient evaluations.}
%     \label{fig:geval_manualeps}
% \end{figure}

% \begin{figure}
%     \centering
%     \includegraphics[width=0.5\linewidth]{graph/ablation/exetime_manualeps.pdf}
%     \caption{A2GD with different choices of $\varepsilon$, in terms of in terms of running time.}
%     \label{fig:time_manualeps}
% \end{figure}

\begin{figure}[t]
    \centering

    \begin{subfigure}{0.46\textwidth}
        \centering
        \includegraphics[width=\linewidth]{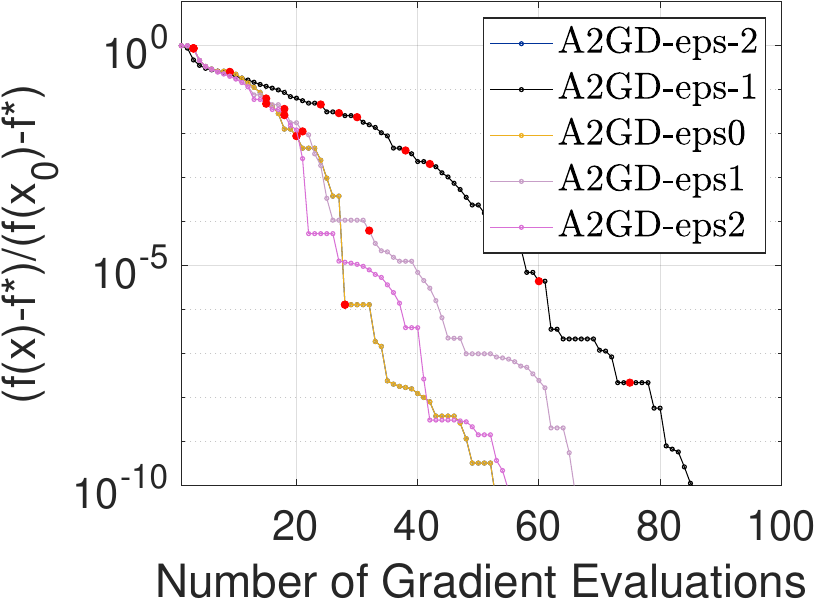}
        % \caption{First}
    \end{subfigure}
    \hfill
    \begin{subfigure}{0.46\textwidth}
        \centering
        \includegraphics[width=\linewidth]{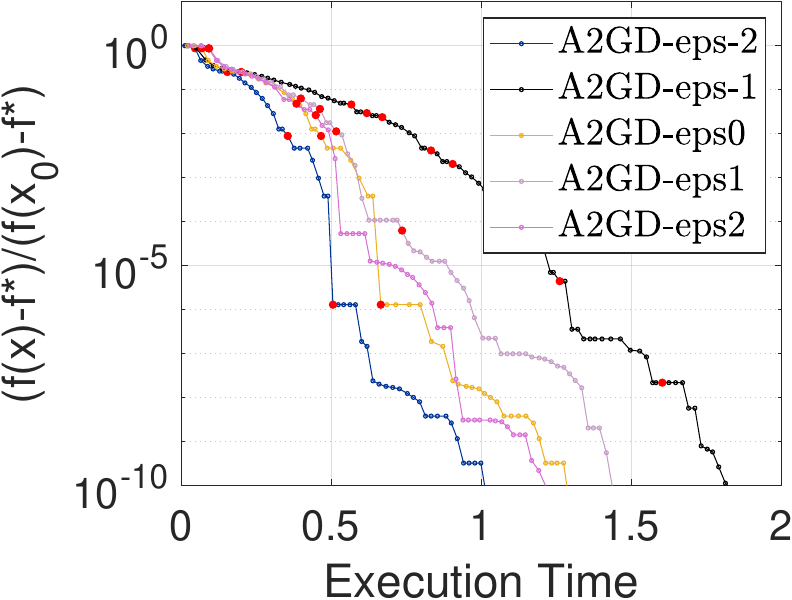}
        % \caption{Second}
    \end{subfigure}

    \caption{A2GD with different choices of $\varepsilon$, in terms of number of gradients (left) and execution time (right).}
    \label{fig:ablation_manualeps}
\end{figure}

\paragraph{Further Discussion on the Problem Scaling}

Apart from the linear example which we discussed in main context, we also tested a range of $\lambda$ values ($\lambda=10^{-2}, 1, 10^2$) in the regularized logistic regression problem. In this case, the $\sqrt{\kappa}$ scaling is less apparent because our method performs very well when $\lambda$ is close to $0$. Even in the case $\lambda = 10^{-2}$, the adaptive $\mu_k$ does not necessarily remain small during the iterations, which can effectively improve the convergence beyond what the nominal condition number would suggest. Details are in Figure \ref{fig:ablation_manuallambda}.

\begin{figure}[t]
    \centering

    \begin{subfigure}{0.46\textwidth}
        \centering
        \includegraphics[width=\linewidth]{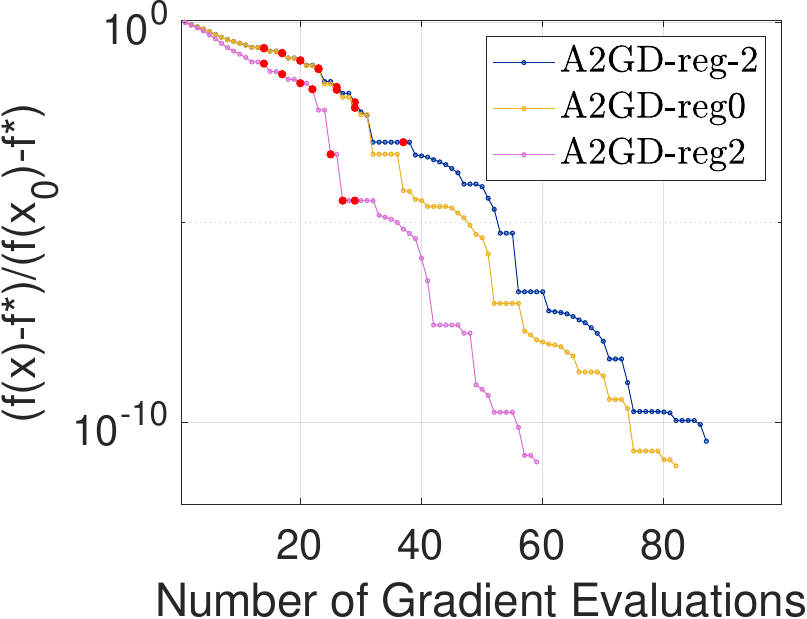}
        % \caption{First}
    \end{subfigure}
    \hfill
    \begin{subfigure}{0.46\textwidth}
        \centering
        \includegraphics[width=\linewidth]{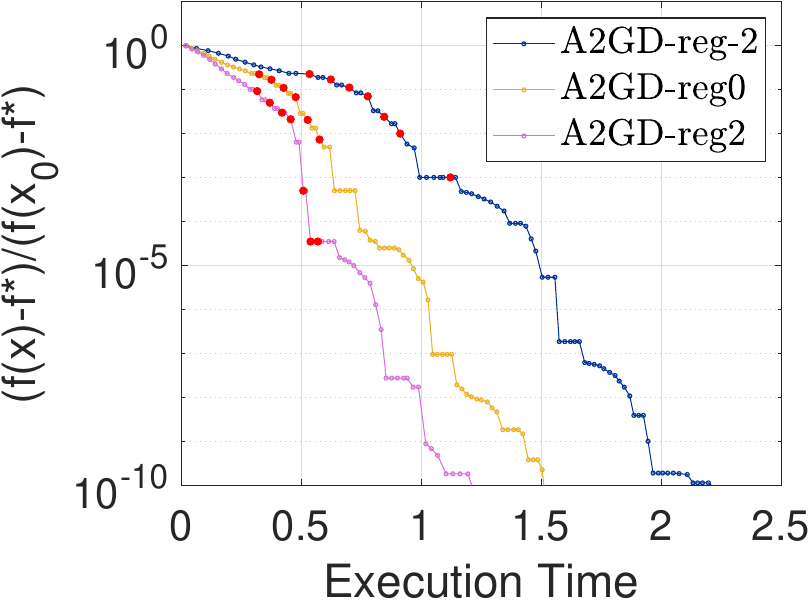}
        % \caption{Second}
    \end{subfigure}

    \caption{A$^2$GD on regularized logistic regression problem with different regularization constants, in terms of number of gradients (left) and execution time (right).}
    \label{fig:ablation_manuallambda}
\end{figure}

% \begin{figure}
%     \centering
%     \includegraphics[width=0.5\linewidth]{graph/ablation/geval_difflambda.pdf}
%     \caption{A2GD on regularized logistic regression problem with different regularization constants, in terms of gradient evaluations.}
%     \label{fig:difflambda-geval}
% \end{figure}

% \begin{figure}
%     \centering
%     \includegraphics[width=0.5\linewidth]{graph/ablation/exetime_difflambda.pdf}
%     \caption{A2GD on regularized logistic regression problem with different regularization constants, in terms of running time.}
%     \label{fig:difflambda-exetime}
% \end{figure}

\paragraph{An Empirical Study of Adaptive Degree of Freedom}
As mentioned in the Introduction Section, we can classify accelerated gradient methods via the adaptive degree of freedom (ADoF). ADoF-$0$ simply means non-adaptive accelerated methods, for example NAG. ADoF-$1$ and ADoF-$2$ mean the method has $1$ and $2$ adaptive parameters respectively. For example, the AcceleGrad method proposed in \citet{levy2018online} belongs to ADoF-$1$, while our A$^2$GD method and a few other baselines in the main context belong to ADoF-$2$. To see their difference numerically, we compare the 3 methods on the regularized logistic regression problem on Adult Census Income dataset (Figure \ref{fig:ablation_ADoF}). As neither AcceleGrad nor NAG requires restarting, we also use A$^2$GD-plain for a fair comparison. AcceleGrad performs slightly better than NAG by reducing the oscillations and converging in a reasonable pace, but its improvement is limited due to the semi-adaptivity. In contrast, A$^2$GD outperforms AcceleGrad and NAG dramatically. This can be seen as an example where the additional adaptive parameter (in A$^2$GD, the parameter is $\mu$) brings much improvement. 

\begin{figure}[t]
    \centering

    \begin{subfigure}{0.46\textwidth}
        \centering
        \includegraphics[width=\linewidth]{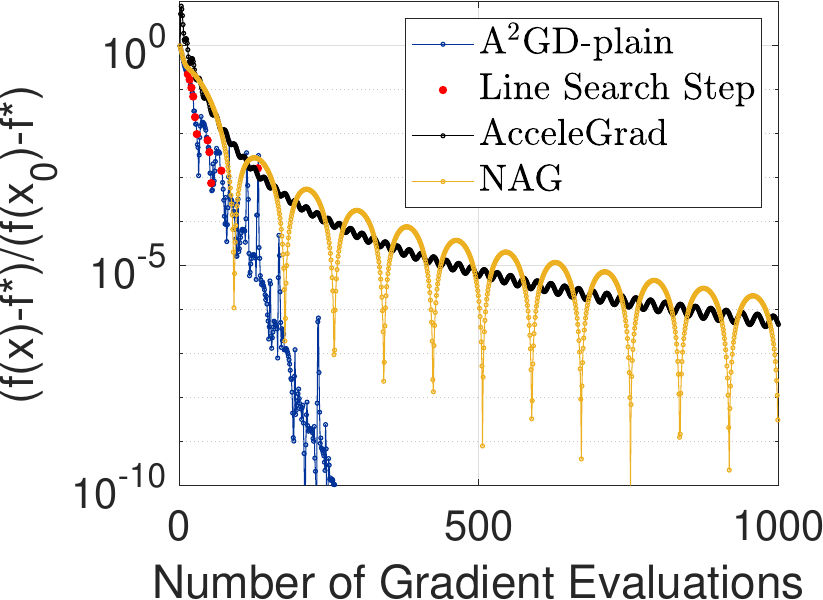}
        % \caption{First}
    \end{subfigure}
    \hfill
    \begin{subfigure}{0.46\textwidth}
        \centering
        \includegraphics[width=\linewidth]{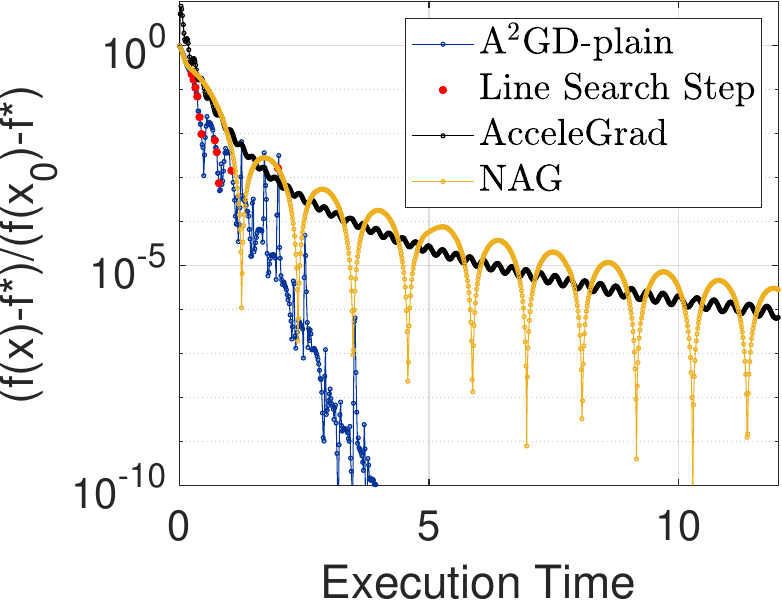}
        % \caption{Second}
    \end{subfigure}

    \caption{Comparison of methods with adaptive degree of freedom $0,1,2$. Results are in terms of number of gradients (left) and execution time (right).}
    \label{fig:ablation_ADoF}
\end{figure}

\section*{\RV{Appendix E: Convergence Graphs in terms of execution Time}}

Here, we present the convergence graphs in which the x-axis represents execution time (unit: second). As discussed, for all tested gradient methods, execution time is approximately proportional to the number of gradients, but we still present them here for completeness.

\paragraph{Regularized Logistic Regression}
Below are the results for the regularized logistic regression problem.\\
\begin{minipage}[htdp]{0.31\textwidth}
    \centering
    \includegraphics[width=0.975\linewidth]{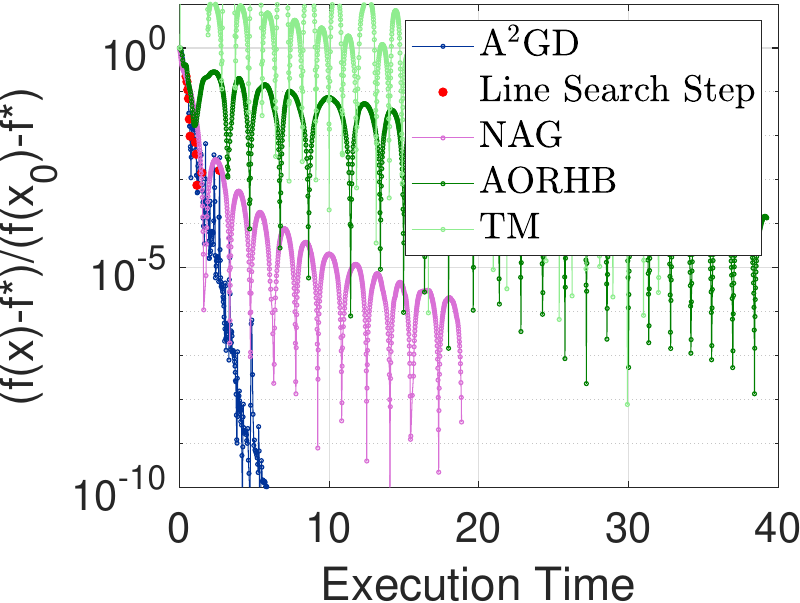}
    \captionsetup{hypcap=false}
    \captionof{figure}{Comparison without restarting.}      \label{fig:lor_nonadaptive_norestart-time}
\end{minipage}%
\hfill
\begin{minipage}[htdp]{0.31\textwidth}
    \centering
    \includegraphics[width=0.975\linewidth]{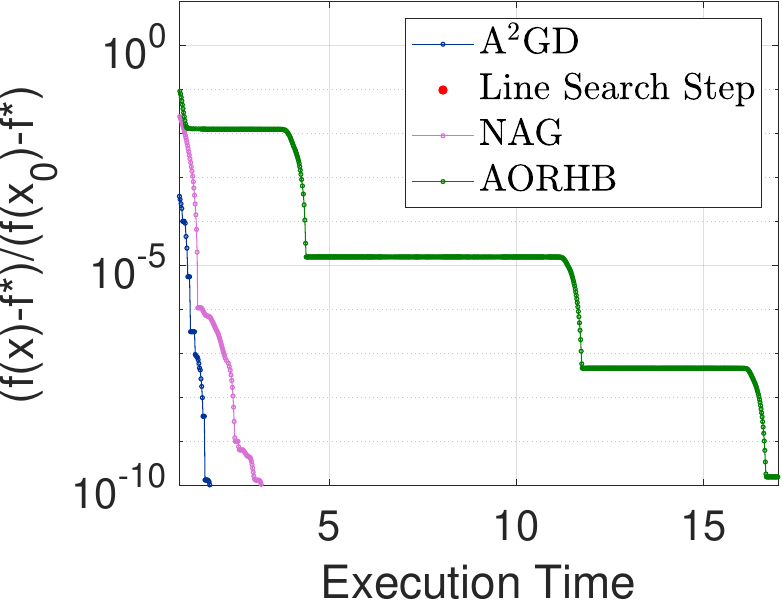}
    \captionsetup{hypcap=false}
    \captionof{figure}{Comparison with restarting.
} 
    \label{fig:lor_nonadaptive_restart-time}
\end{minipage}
\hfill
\begin{minipage}[htdp]{0.31\textwidth}
    \centering
    \includegraphics[width=0.95\linewidth]{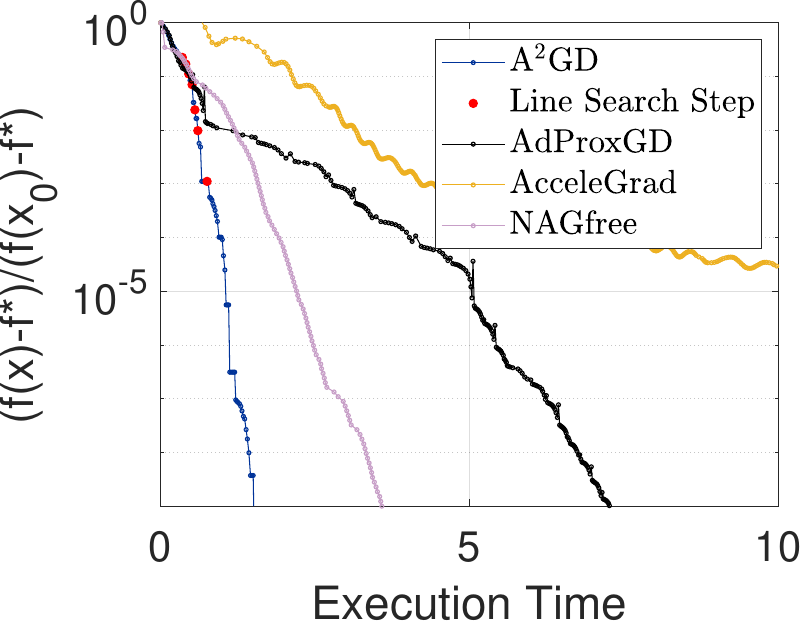}
    \captionsetup{hypcap=false}
    \captionof{figure}{A$^{2}$GD compared to other adaptive methods.
} 
    \label{fig:lor_adaptive-time}
\end{minipage}

\paragraph{Maximum Likelihood Estimation}
Below are the results for the maximum likelihood estimation problem.\\

\begin{minipage}[htdp]{0.45\textwidth}
    \centering
    \includegraphics[width=0.75\linewidth]{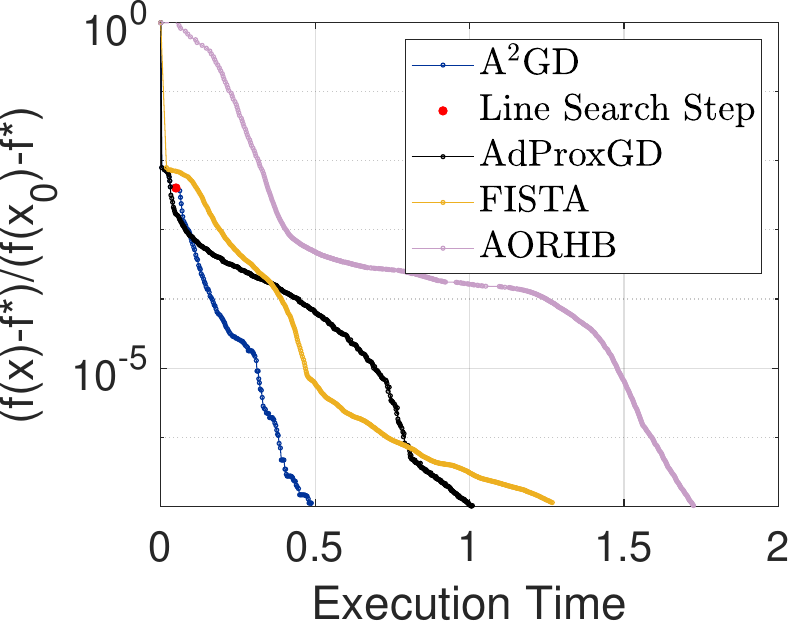}
    \captionsetup{hypcap=false}
    \captionof{figure}{Error curves under setting (1).}      \label{fig:mle_easy-time}
\end{minipage}%
\hfill
\begin{minipage}[htdp]{0.45\textwidth}
    \centering
    \includegraphics[width=0.75\linewidth]{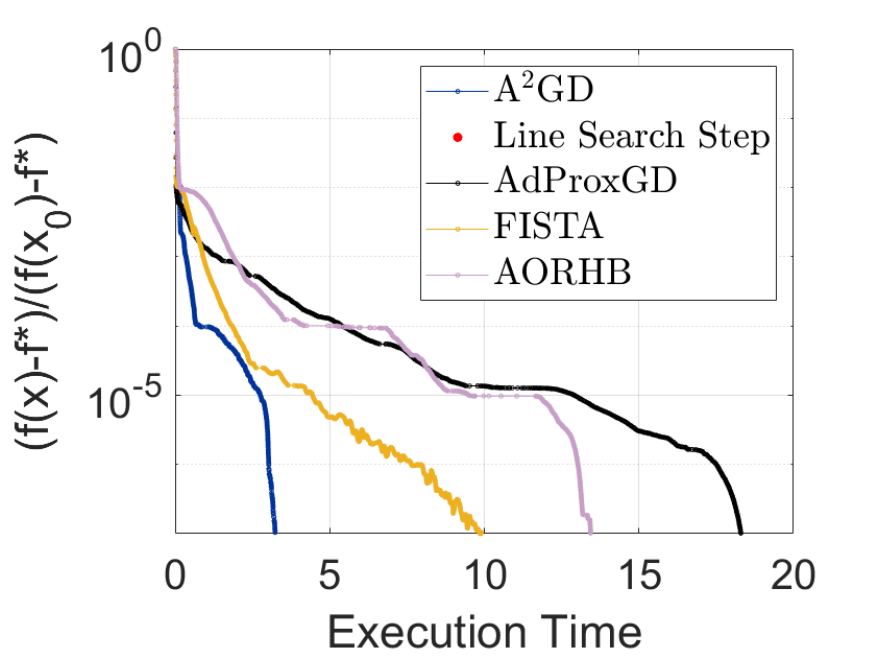}
    \captionsetup{hypcap=false}
    \captionof{figure}{Error curves under setting (2).
} \label{fig:mle_hard-time}
\end{minipage}

\paragraph{$\ell_{1-2}$ Nonconvex Minimization} Below are the results for the $\ell_{1-2}$ nonconvex minimization problem.\\

\begin{figure}[htbp]
    \centering
        \label{fig:LASSO_500-1000d-time}
\includegraphics[width=0.375\linewidth]{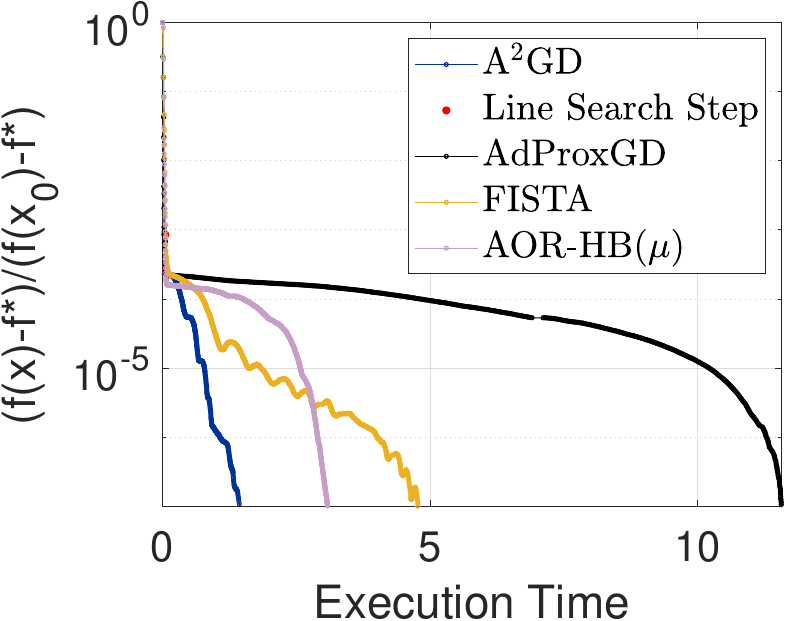}
\captionof{figure}{Error curve for $\ell_{1\text{-}2}$ problem with $n=500, p=1000$. 
%(\RV{figure 15 to be deleted}) 
%A$^2$GD (blue curve) converges much faster than AdProxGD (black curve) and FISTA (yellow curve).
}
\end{figure}
\section*{LLM usage}
In preparing this manuscript, large language models (LLMs) were employed exclusively to assist with language-related tasks, such as improving readability, grammar, and style. The models were not used for research ideation, development of methods, data analysis, or interpretation of results. All scientific content, including problem formulation, theoretical analysis, and experimental validation, was conceived, executed, and verified entirely by the authors. The authors bear full responsibility for the accuracy and integrity of the manuscript.
% \section{Alternative proof}
% \input{appendix_D}

\section*{Ethics statement}
This work is purely theoretical and algorithmic, focusing on convex optimization methods. It does not involve human subjects, sensitive data, or applications that raise ethical concerns related to privacy, security, fairness, or potential harm. All experiments are based on publicly available datasets or synthetic data generated by standard procedures. The authors believe that this work fully adheres to the ICLR Code of Ethics.

\section*{Reproducibility statement}
We have taken several measures to ensure the reproducibility of our results. All theoretical assumptions are explicitly stated, and complete proofs are provided in the appendix. For the experimental evaluation, we describe the setup, parameter choices, and baselines in detail in the main text. The source code for our algorithms and experiments are available as supplementary materials. Together, these resources should allow others to reproduce and verify our theoretical and empirical findings.

\end{document}